\documentclass[a4paper, 11pt]{article}

\usepackage[
            includefoot,  %Uncomment to put page number above margin
            marginparwidth=0in,     % Length of section titles
            marginparsep=0in,       % Space between titles and text
            margin=1.45in,               % 1 inch margins
            includemp]{geometry}

\usepackage{bbm}
\usepackage{float}
\usepackage{graphicx}                  % derni\`ere \'etant la langue principale
\usepackage{amssymb}
\usepackage{amsfonts}
\RequirePackage{amsmath}
\RequirePackage{amsthm}

\RequirePackage{latexsym}
\RequirePackage{amsfonts}
\RequirePackage{amsmath}
\RequirePackage{amssymb}
\RequirePackage{amsthm}

\RequirePackage[backrefs]{amsrefs}

\newtheoremstyle{montheoreme}% name
  {}%				Space above
  {}% 			Space below
  {\itshape}%		Body font
  {}%				Indent amount (empty = no indent, \parindent = para indent)
  {\bf}%			Thm head font
  {.}%			Punctuation after thm head
  {.5em}%			Space after thm head: " " = normal interword space;
  {}%				Thm head spec (can be left empty, meaning `normal')
\newtheoremstyle{maremarque}% name
  {}%				Space above
  {}% 			Space below
  {}%				Body font
  {}%				Indent amount (empty = no indent, \parindent = para indent)
  {\bf}%			Thm head font
  {.}%			Punctuation after thm head
  {.5em}%			Space after thm head: " " = normal interword space;
  {}%				Thm head spec (can be left empty, meaning `normal')
\usepackage{fancyhdr}
\theoremstyle{montheoreme}

\newtheorem{thm}{Theorem}[section]
\newtheorem{defn}[thm]{Definition}
\newtheorem{prop}[thm]{Proposition}
\newtheorem{lem}[thm]{Lemma}

\newtheorem{cor}[thm]{Corollary}
\theoremstyle{maremarque}
\newtheorem*{rmq}{Remark}{}

\DeclareMathOperator{\Ent}{Ent}

\newcommand{\T}{\mathbb{T}}

\makeatletter

\@addtoreset{equation}{chapter}
\makeatother

\makeatletter

\@addtoreset{equation}{section}
\makeatother

\begin{document}
%-------------------------------------PAGE DE TITRE------------------------------
\title{A two-scale approach to the hydrodynamic limit \newline Part II: local Gibbs behavior}

\vspace{5mm}

\author{Max Fathi \thanks{LPMA, University Paris 6, France, mfathi@clipper.ens.fr.} }
\date{\today}

\maketitle

%----------------------------------------------------RESUME ET REMERCIEMENTS---------------------

\begin{abstract}
This work is a follow-up on [GOVW]. In that previous work a two-scale approach was used to prove the logarithmic Sobolev inequality for a system of spins with fixed mean whose potential is a bounded perturbation of a Gaussian, and to derive an abstract theorem for the convergence to the hydrodynamic limit. This strategy was then successfully applied to Kawasaki dynamics. Here we shall use again this two-scale approach to show that the microscopic variable in such a model behaves according to a local Gibbs state. As a consequence, we shall prove the convergence of the microscopic entropy to the hydrodynamic entropy.

\end{abstract}
\tableofcontents
\newpage

%------------------------------------------------------TEXTE----------------------------------------------------------------------------------------------------------------------
\large \textbf{Introduction} 

\vspace{0.2cm}

A local Gibbs measure is a vague term used to designate a measure whose density (with respect to the plain Gibbs measure) takes the form $G(x) = \exp(\sum \lambda_ix_i)$, where the coefficients $\lambda_i$ vary ``at macroscopic scale". They have been used by Guo, Papanicolaou and Varadhan in [GPV] for the Ginzburg-Landau model, and also play a crucial role in the relative entropy method devised by Yau in [Y]. They represent in some sense a ``typical" microscopic distribution having the correct hydrodynamic profile. The main result in [Y] can be informally summarized as follows: if the initial datum is in local Gibbs state, then at later times the microscopic variable is very close (in the sense of Kullback information) to be in local Gibbs state too. The local Gibbs state Yau used is defined in terms of the hydrodynamic equation, and chosen so that it is close to being a solution of the microscopic equation.

\vspace{0.1cm}

\hspace{0.3cm} In a more recent contribution, Kosygina [K] proved that the solution of the Ginzburg--Landau model behaves like a local Gibbs state for all positive times even if it does not at initial time. That is, there is a time-dependent family of vectors $\lambda(t)$ such that, if $f(t,\cdot)$ is the density at time $t$ with respect to the equilibrium measure $\mu$ of a system of N continuous spins $x_i$ interacting according to Kawasaki dynamics, then the relative entropy of $f\mu$ with respect to the measure $\nu(dx) = \frac{1}{Z}\exp(\lambda \cdot x)\mu(dx)$, given by
$$Ent_{\nu}(f\mu) = \int{\rho \log \rho \hspace{1mm}},$$
with $\rho$ being the density of $f \mu$ with respect to $\nu$, goes to $0$ when $N$ goes to infinity for any time $t > 0.$ The equilibrium measure $\mu(dx) := \exp(-H(x))dx$ is assumed to have a Ginzburg--Landau type potential, that is
$$H(x) = \sum \psi(x_i),$$
where $\psi$ is the single-site potential. Kosygina's proof relied on the logarithmic Sobolev inequality, and she used an assumption of uniform convexity of the Ginzburg--Landau potential $\psi$. In the present work we shall generalize these results to cover a certain class of nonconvex potentials. At the same time we shall point out the role of another information-theoretical inequality, the so-called HWI inequality introduced by Otto and Villani in [OV]. This inequality will allow us to pass from a convergence in a Wasserstein distance sense to a convergence in relative entropy. To be used efficiently in this setting, the HWI inequality needs a log-concave reference measure, which is not the case for the microscopic equilibrium measure when $\psi$ is not convex. This is why we shall use, like in [GOVW], the convexification induced by the macroscopic block decomposition. Since we will then need our local Gibbs state to be compatible with the passage to macroscopic scale, we will use a local Gibbs state slightly different from the one used in [Y] (the value of $\lambda \cdot x$ must only depend on the macroscopic profile associated with $x$), but such that when $N$ goes to infinity, the relative entropy with respect to either measure behaves in the same way.

\hspace{0.3cm} As in [K], we shall also prove the (physically relevant) convergence of the microscopic entropy to the macroscopic (hydrodynamic) entropy, that is 
$$\frac{1}{N}\int{f(t,x)\log f(t,x) \mu(dx)} \rightarrow \int_{\mathbb{T}}{\varphi(\zeta(t,\theta)) d\theta} - \varphi\left(\int_{\mathbb{T}}{\zeta(t,\theta) d\theta}\right),$$
where $\zeta$ is the hydrodynamic limit, and $\varphi$ is the Cram\'er transform of the potential. However, we shall deduce it from the local Gibbs behavior, while Kosygina does it the other way round.

\hspace{0.3cm} Our two-scale approach will only yield convergences in $\textsl{L}^1(dt)$. However, by using a method of [K] in conjunction with these results, we will be able to immediately prove that the convergence is uniform in time, as long as we stay away from time $t = 0$.

\hspace{0.3cm} It should be noted that, while Kosygina used an assumption of convexity, it was mainly required to ensure that the equilibrium measures satisfies a logarithmic Sobolev inequality. In light of the recent work [MO], it seems that her method can be adapted to the class of nonconvex potentials covered here. On the other hand, our method does not cover potentials with superquadratic growth. This restriction is inherited from the results in [GOVW]. If the two scale approach could be extended to cover superquadratic potentials, our method could also be extended, with minor technical modifications. However, even with the logarithmic Sobolev inequality obtained in [MO], extending the section of [GOVW] that concerns the hydrodynamic limit to superquadratic potentials is nontrivial, and remains to be done.

\hspace{0.3cm} The plan of this paper is as follows : in Section 1, we will recall the framework and results of [GOVW] which will be used in this article and present our main results, in both the abstract framework and their application for Kawasaki dynamics. Sections 2 and 3 will then give the details of the proofs.
\vspace{0.5cm}

\large \textbf{Notation}

\vspace{0.2cm}
- $\nabla$ stands for the gradient, $\operatorname{Hess}$ for Hessian, $|\cdot |$ for norm and $\langle \cdot, \cdot \rangle$ for inner product. Whenever necessary, the space to which these are associated will be indicated with a subscript.

- $A^t$ is the transpose of the operator $A$.

- Ran($A$) is the range of the operator $A$.

- $\Phi(x) = x\log x$.

- $\operatorname{Ent}_{\mu}(f) := \int{\Phi(f)d\mu} - \Phi(\int{f d\mu})$ is the (negative of the) entropy of the positive function $f$ with respect to the probability measure $\mu$. 

- C is a positive constant, which may change from line to line, or even within a line.

- Z is a positive constant enforcing unit mass of a given probability measure.

- $id_E$ is the identity map $E \rightarrow E$

- LSI is an abbreviation for Logarithmic Sobolev Inequality.

- $\Gamma(Y, |\cdot|_Y) := \int{\exp(-|y|_Y^2/2)dy}$ is the Gaussian integral on the space $Y$ with respect to the norm $|\cdot|_Y$.

- $C^{1,2}(A \times B)$ is the space of real-valued functions on $A \times B$ which are $C^1$ with respect to the first variable and $C^2$ with respect to the second variable.

- $W_2(\mu, \nu)$ is the Wasserstein distance between two probability measures $\mu$ and $\nu$ with finite second moment. It is defined through the formula $W_2(\mu, \nu)^2 := \underset{\pi \in \Pi}{\inf} \int{|x - y|^2\pi(dx, dy)}$, where $\Pi$ is the set of all coupling of $\mu$ and $\nu$;

- $I_{\mu}(\nu)$ is the Fisher information of the probability measure $\nu$ with respect to $\mu$, given by $\int{\frac{|\nabla f|^2}{f}d\mu}$ if $\nu = f\mu$ for some density $f$, and $+\infty$ if not.

\vspace{1cm}

\section{Background and Main Results}

The aim of this section is to recall the setting and the main results of [GOVW], as well as to present the new results brought by the present paper.

\subsection{Logarithmic Sobolev inequalities}

Throughout this work, $X$ and $Y$ are two Euclidean spaces. 
It is convenient to think of $X$ as the space of microscopic variables, 
and $Y$ as the space of macroscopic variables. 
We consider a linear operator $P : X \rightarrow Y$, that associates to the microscopic profile $x$ the corresponding macroscopic profile $y = Px$. 
We shall assume that there is an integer $N \in \mathbb{N}$, which measures the size of the microscopic system, such that
\begin{equation} 
\label{2} P N P^t = id_Y. 
\end{equation}
We shall keep the same framework as in [GOVW], by considering a measure $\mu(dx) = \exp(-H(x))\hspace{0.1cm}dx$ on $X$, and its decomposition, as induced by the operator $P$. The measure $\bar{\mu} = P_{\#}\mu$ is the distribution of the macroscopic profile, and for all $y \in Y$, $\mu(dx|y)$ is the conditional distribution of $x$ given $Px = y$. This decomposition induces a natural coarse-graining of the microscopic Hamiltonian $H$, defined by $\bar{H}(y) := -\frac{1}{N}\log\left(\frac{d\bar{\mu}}{dy}\right)$, so that 
$$\bar{\mu}(dy) = \exp(-N\bar{H}(y))\hspace{0.1cm}dy.$$

One of the tools frequently used to study particle systems is a logarithmic Sobolev inequality. Let us first recall the definition: 

\begin{defn} Let $X$ be a Riemannian manifold. A probability measure $\mu$ on $X$ is said to satisfy a LSI with constant $\rho > 0$ if, for any locally Lipschitz, nonnegative function $f \in \textsl{L}^1(\mu)$,
$$\int{f \log (f) d\mu} - \left(\int{f d\mu}\right)\log \left(\int{f d\mu}\right) \leq \frac{1}{\rho}\int{\frac{|\nabla f|^2}{2f}d\mu}.$$
\end{defn}

There are many criterions and applications for LSI in the literature. [L] contains a nice introduction to the topic. One of the main results of [GOVW] is the following sufficient condition for LSI, based on the two-scale decomposition of $\mu$.

\begin{thm} [Two-scale LSI]

\noindent Let $\mu(dx) =$ $\exp(-H(x))dx$ be a probability measure on $X$, and let $P : X \rightarrow Y$ satisfy (\ref{2}). Assume that

(i) $\kappa :=$ 
\begin{equation} \label{5}
\underset{x \in X}{\max} \left\{\left\langle \operatorname{Hess} H(x) \cdot u, v\right\rangle, 
u \in \operatorname{Ran}\left(NP^tP\right), v \in \operatorname{Ran}\left(id_X - NP^tP\right), |u| = |v| =1\right\} 
\end{equation}
is finite;

(ii) There is $\rho > 0$ such that $\mu(dx|y)$ satisfies LSI($\rho$) for all $y$;

(iii) There is $\lambda > 0$ such that $\bar{\mu}$ satisfies LSI($\lambda N$).

Then $\mu$ satisfies LSI($\hat{\rho}$), with 
\begin{equation} \label{7}
\hat{\rho} := \frac{1}{2}\left(\rho + \lambda + \frac{\kappa^2}{\rho} - \sqrt{(\rho + \lambda +  \frac{\kappa^2}{\rho})^2 - 4\rho \lambda}\right) > 0.
\end{equation}

\end{thm}

\subsection{Hydrodynamic limits}
Let us now recall the setting of the abstract criterion for hydrodynamic limits proved in [GOVW]. 
We endow the space $X$ with a Gibbs probability measure $\mu$, and we consider a positive definite symmetric linear operator $A : X \rightarrow X$. 
The stochastic dynamics on $X$ that is studied is described by the equation 
\begin{equation} \label{8}
\frac{\partial}{\partial t}(f\mu) = \nabla \cdot \left(A \nabla f \mu\right).
\end{equation}
This equation is to be understood in a weak sense. That is, for all smooth test function $\xi$, we have $\frac{d}{dt}\int{\xi(x)f(t,x)\mu(dx)} = -\int{\nabla \xi(x) \cdot A\nabla f(t,x)\mu(dx)}$. 
Given an initial condition $f(0, \cdot)$ such that $f(0,x)\mu(dx)$ is a probability measure, 
the solution $f(t,x)$ is at all times the microscopic density of a probability measure with respect to $\mu$. 

The aforementioned abstract criterion states that, under certain conditions, 
and in a precise sense, the macroscopic profile $y = Px$, with law given by $\bar{f}(t,y) = \int_{\{Px = y\}}{f(t,x)\mu(dx)}$ is close to the solution of the following differential equation : 
\begin{equation} \label{9}
\frac{d\eta}{dt} = -\bar{A}\nabla \bar{H}(\eta(t))
\end{equation}
where $\bar{A}$ is the symmetric, positive definite operator on $Y$ defined by 
\begin{equation} 
\bar{A}^{-1} = PA^{-1}NP^t.
\end{equation}

We can now recall the abstract theorem proved in [GOVW] :

\begin{thm} \label{10}
Let $\mu(dx) = \exp(-H(x))dx$ be a probability measure on $X$, and let $P : X \rightarrow Y$ satisfy (\ref{2}). We define $M := \operatorname{dim} Y + 1$. Let $A : X \rightarrow X$ be a symmetric, definite positive operator, and $f(t,x)$ and $\eta(t)$ be the solutions of (\ref{8}) and (\ref{9}), with initial data $f(0,\cdot)$ and $\eta_0$ respectively. Assume that :

(i) $\kappa$ as defined by (\ref{5}) is finite;

(ii) There is $\rho > 0$ such that $\mu(dx|y)$ satisfies LSI($\rho$) for all $y$;

(iii) There is $\lambda > 0$ such that $\langle \operatorname{Hess} \bar{H} (y)\tilde{y}, \tilde{y} \rangle_Y \geq \lambda \langle \tilde{y}, \tilde{y} \rangle_Y$ for all $y$, $\tilde{y} \in Y$;

(iv) There is $\alpha > 0$ such that $\int_X{|x|^2\mu(dx)} \leq \alpha N$;

(v) There is $\beta > 0$ such that $\underset{y \in Y}{\inf} \bar{H}(y) \geq -\beta$;

(vi) There is $\gamma > 0$ such that for all $x \in X$, $$|(id_X - NP^tP)x|^2 \leq \gamma M^{-2} \langle x, Ax \rangle_X;$$

(vii) There are constants $C_1$ and $C_2$ such that the initial datum satisfy $$\int{f(0,x)\log f(0,x) \mu(dx)} \leq C_1 N \hspace{0.3cm}  \text{and}  \hspace{0,3cm}  \bar{H}(\eta_0) \leq C_2.$$

\vspace{0.2cm}

Define $$\Theta(t) := \frac{1}{2N}\int{\left\langle (x - NP^t\eta(t)), A^{-1}(x - NP^t\eta(t))\right\rangle f(t,x)\mu(dx)}.$$

Then for any $T > 0$, we have, with $\hat{\rho}$ given by (\ref{7}), 

$\max \left\{\underset{0 \leq t \leq T}{\sup} \Theta(t), \frac{\lambda}{2}\int_0^T{\left(\int_Y{|y-\eta(t)|_Y^2\bar{f}(t,y)\bar{\mu}(dy)}\right)dt}\right\}$

$\leq \Theta(0) + T\left(\frac{M}{N}\right) + \frac{1}{M^2}\left(\frac{C_1\gamma \kappa^2}{2\lambda \rho^2}\right) + \frac{1}{M}\left[\sqrt{2\gamma T}\left(\alpha + \frac{2C_1}{\hat{\rho}}\right)^{1/2}(\sqrt{C_1} + \sqrt{C_2 + \beta})\right]$

$=: \Xi(T,M,N).$
\end{thm}

This theorem means that, if we consider a sequence of data 

$\{X_{\ell}, Y_{\ell}, N_{\ell}, P_{\ell}, A_{\ell}, \mu_l, f_{0,\ell}, \eta_{0,\ell}\}_{\ell \in \mathbb{N}}$ that satisfies the previous assumptions with uniform constants, 
and if we assume that
\begin{equation} M_{\ell} \uparrow \infty; \hspace{0.5cm} N_{\ell} \uparrow \infty; \hspace{0.5cm} \frac{N_{\ell}}{M_{\ell}} \uparrow \infty
\end{equation}
and that the initial data $\Theta_{\ell}(0)$ goes to $0$, then for all $T > 0$ we have
\begin{equation} \label{eq_convergence}
\underset{\ell \uparrow \infty}{\lim} \hspace{0.1cm} \underset{0 \leq t \leq T}{\sup} \frac{1}{N_{\ell}}\int{(x-N_{\ell}P_{\ell}^t\eta_{\ell}(t))\cdot A_{\ell}^{-1}(x-N_{\ell}P_{\ell}^t\eta_{\ell}(t))f_{\ell}(t,x)\mu_{\ell}(dx)}= 0
\end{equation}
and
$$\underset{\ell \uparrow \infty}{\lim} \int_0^T{\int_Y{|y-\eta_{\ell}(t)|_Y^2\bar{f}_{\ell}(t,y)\bar{\mu}(dy)}dt} = 0.$$

\begin{rmq} As noted in [GOVW], hypothesis (iii) of this theorem implies hypothesis (iii) of 
Theorem \textbf{\ref{5}} by the Bakry--\'Emery theorem, a proof of which can be found in [L].
\end{rmq}

Using this result, we will deduce bounds on the relative entropy with respect to a well-chosen local Gibbs state. Let us first give a precise definition of what we mean by a local Gibbs state.

\begin{defn} Let $\eta \in Y$. The local Gibbs state associated with $\eta$ is the probability measure on $X$ whose density is given by
\begin{equation} \label{3} 
G^{\eta}(x)\mu(dx) = Z^{-1}\exp\left(\vec{\lambda} \cdot x \right)\mu(dx), \hspace{0.5cm} \vec{\lambda} = NP^t\nabla \bar{H}(\eta).
\end{equation}
\end{defn}

\begin{rmq} Notice that, in this definition, $G^{\eta}(x)$ only depends on the macroscopic profile $Px$. This differs from the local Gibbs measure used in [Y], which (slowly) varied at the microscopic scale. But here, we force the maximum of the macroscopic probability density to be reached at $\eta$, which makes this definition convenient.
\end{rmq}

We can now formulate our results in this abstract setting:  

\begin{thm} \label{4}
Let $G(t,\cdot) = G^{\eta(t)}$ denote the local Gibbs state associated with $\eta(t)$, where $\eta(t)$ solves the macroscopic equation (\ref{9}). Suppose assumptions (i) to (vii) from Theorem \textbf{\ref{10}} hold. Further assume that

(viii) There is $\tau > 0$ such that $A \geq \tau \operatorname{Id}_X$

(ix) The Hessian of $\bar{H}$ is bounded above, i.e. there exists $\Lambda > 0$ such that for all $y \in Y$ we have $\operatorname{Hess} \bar{H}(y) \leq \Lambda \operatorname{Id}$;

Then  

(a) The relative entropy with respect to the local Gibbs state is controlled as follows:  
\begin{equation}\int_0^T{\frac{1}{N}\int{\Phi\left(\frac{f(t,x)}{G(t,x)}\right)G(t,x) \mu(dx)}dt} = \text{\Large O} \left(\sqrt{\Theta(0) + \frac{M}{N} + \frac{1}{M}} \right)
\end{equation}
where the actual constants in the bound depend on $T$, $\lambda$, $\alpha$, $\gamma$, $\rho$, $\kappa$, $\tau$, $C_1$ and $C_2$, but not on $M$ and $N$;

(b) The difference between the microscopic free energy and the free energy associated with the macroscopic profile $\eta$ is bounded as follows:
\begin{align}
\int_0^T&{\left|\frac{1}{N}\int{\Phi(f(t,x))\mu(dx)} - \bar{H}(\eta(t))\right|dt} \notag \\
&= \text{\Large O} \left(\sqrt{\Theta(0) + \frac{M}{N} + \frac{1}{M}} \right) \notag \\
&+ \text{\Large O} \left(\frac{M}{N}\right) \times \max\left(\left|\log\left(\frac{\Gamma(Y,|\cdot|_Y)^{2/(M-1)}}{\Lambda N}\right)\right|\left|\log\left(\frac{\Gamma(Y,|\cdot|_Y)^{2/(M-1)}}{\lambda N}\right)\right|\right)
\end{align}
\end{thm}

\begin{rmq} [On the assumptions]
Assumption $(viii)$ is always true, since we assumed $A$ to be a positive symmetric operator on $X$, but I write it down this way because, in the next Corollary, I will require this lower bound to be uniform in $N$, and setting it this way makes this requirement clearer. When the Hessian of $H$ is bounded above (which will be the case in the next section for the application to Kawasaki dynamics), both assumptions $(i)$ and $(ix)$ will be satisfied. As for $\Gamma(Y,|\cdot|_Y)$, it will have a nice behavior when $|\cdot|_Y$ is comparable to the $\textsl{L}^2$ norm, as we shall see in the proof of Theorem \ref{30}.
\end{rmq}

With this theorem, we can obtain quantitative controls in the hydrodynamic limit: 

\begin{cor} \label{31}
Consider a sequence of data $\{X_{\ell}, Y_{\ell}, N_{\ell}, P_{\ell}, A_{\ell}, \mu_{\ell}, f_{0,\ell}, \eta_{0,\ell}\}$ satisfying the previous assumptions, with uniform constants $\alpha, \lambda, \beta, C_1, C_2, \tau$ and $\Lambda$. Assume moreover that
\begin{equation} \label{601} N_{\ell} \rightarrow \infty; \hspace{1cm} M_{\ell} \rightarrow \infty; \hspace{1cm} \frac{M_{\ell}}{N_{\ell}} \rightarrow 0;
\end{equation}
\begin{equation} \label{602}
\frac{M_{\ell}}{N_{\ell}}\log\left(\frac{\Gamma(Y_{\ell},|\cdot|_Y)^{1/(M-1)}}{N_{\ell}}\right) \rightarrow 0
\end{equation}
and that the sequence of initial data satisfies 
$$\Theta_{\ell}(0) \rightarrow 0.$$
Then we have, for all $T > 0,$

(a') $$\int_0^T{\frac{1}{N}\int{\Phi\left(\frac{f_{\ell}(t,x)}{G_{\ell}(t,x)}\right)G_{\ell}(t,x) \mu_{\ell}(dx)}dt} \longrightarrow 0;$$

(b') $$\int_0^T{\left|\frac{1}{N}\int{\Phi(f_{\ell}(t,x))\mu_{\ell}(dx)} - \bar{H_{\ell}}(\eta_{\ell}(t))\right|dt} \longrightarrow 0.$$
\end{cor}

Let us summarize these results in the language of statistical physics :

-The microscopic variables are approximately distributed according to a local Gibbs state, in the sense of relative Kullback information, in a time-integrated sense on $[0,T]$.

-The microscopic free energy converges to the hydrodynamic free energy, in $\textsl{L}^1([0,T])$.

In the next section, in the case of a concrete example, we will reinforce this into a convergence uniformly in time as long as we stay away from zero.

\begin{rmq}
Using the Otto-Villani theorem, which states that the Wasserstein distance $W_2(\nu, \mu)^2$ is controlled by the entropy $\operatorname{Ent}_{\mu}(\nu)$ when $\mu$ satisfies a logarithmic Sobolev inequality (see [OV], or [Go] for an alternate proof), it is possible to show that (a') implies
$$\int_0^T{\frac{1}{N}W_2(f(t)\mu, G(t)\mu)^2 \hspace{1mm} dt} \longrightarrow 0,$$
with the Wasserstein distance associated to the $L^2$ structure, rather than the penalized $A^{-1}$ scalar product that appears in [GOVW]. Since the $L^2$ norm is strictly stronger than the $A^{-1}$ norm, this shows that our convergence in entropy result is strictly stronger than the convergence (\ref{eq_convergence}), as long as we integrate in time. We will later see that this convergence also holds pointwise, for strictly positive times, even if it only holds in the weaker $A^{-1}$ sense at time zero.

Our results also imply that, at macroscopic scale, we have
$$\int{|y - \eta(t)|^2_Y \bar{f}(t,y)\bar{\mu}(dy)} \longrightarrow 0$$
for any time $t > 0$, while this convergence was only proven in a time-integrated sense in [GOVW]. This statement follows from the convergence to $0$ of $W_2(\bar{f}\bar{\mu}, \bar{G}\bar{\mu})$ and $W_2(\bar{G}\bar{\mu}, \delta_{\eta})$, and the triangle inequality for Wasserstein distances.
\end{rmq}

One of the main tools we shall use is the following interpolation inequality, due to Otto and Villani ([OV], Theorem 5): 

\begin{thm} \label{500} Let $\mu(dx) = e^{-H(x)}dx$ be a probability measure on $\mathbb{R}^n$ with a finite moment of order 2 such that $H \in \textsl{C}^2(\mathbb{R}^n)$ and $\operatorname{Hess} H \geq \lambda I_n$, $\lambda \in \mathbb{R}$. Then for any probability measure $\nu$ on $\mathbb{R}^n$ that is absolutely continuous with respect to $\mu$, we have
$$Ent_{\mu}(\nu) \leq W_2(\mu, \nu)\sqrt{I_{\mu}(\nu)} - \frac{\lambda}{2}W_2(\mu, \nu)^2.$$
In particular, if $H$ is convex, then 
\begin{equation} \label{HWI}
\operatorname{Ent}_{\mu}(\nu) \leq W_2(\mu, \nu)\sqrt{I_{\mu}(\nu)}.
\end{equation}
\end{thm}

We refer to the original article [OV] for a proof of this theorem. This will allow us to transform a convergence in Wasserstein distance and a bound on the Fisher information into a convergence of the relative entropy. However, if we apply this result immediately in the microscopic scale, if we use the usual Euclidean structure, the Wasserstein distance between $f \mu$ and the local Gibbs state does not go to zero. And if we use the penalized Euclidean structure $\langle A^{-1} \cdot, \cdot \rangle$, the lower bound on the Hessian will grow too fast, and  the additional term $(\inf \operatorname{Hess} H) W_{A^{-1}, 2}(f \mu, G\mu)$ will go to infinity. So, in order to get rid of the additional term, we will go to macroscopic scale, where the Hessian of $\bar{H}$ is convex, and use inequality (\ref{HWI}).

\begin{rmq} In this context, Kosygina's method would suggest to decompose the macroscopic relative entropy
\begin{align}
\frac{1}{N}\Ent_{\bar{G}\bar{\mu}}(\bar{f}/\bar{G}) &= \frac{1}{N}\int_Y{\tilde{f}\log \tilde{f} dy} + \int{\bar{H}(y)\bar{f}(y)\bar{\mu}(dy)} \notag \\
&+ \frac{1}{N}\log \bar{Z} - \int{\nabla \bar{H}(\eta)\cdot y \bar{f}(y)\bar{\mu}(dy)} \notag
\end{align}
where $\tilde{f} = e^{-N\bar{H}}f$ is the density of the coarse-grained state with respect to the Lebesgue measure. We can get a bound on the time-integral of the sum of the last three terms of the same type as those in Theorem \textbf{\ref{4}}. So the problem would be to bound $\frac{1}{N}\int_Y{\tilde{f}\log \tilde{f} dy}$. 

For the application to Kawasaki dynamics, Kosygina proved a vanishing upper bound on $\frac{1}{N}\int_{t'}^t{\int_Y{\tilde{f(s)}\log \tilde{f(s)} dy}ds}$ for times $t > t' > 0$, which has the same order of magnitude in the system size $N = KM$ as ours. Her proof consists in showing that we can replace the law of our process Kawasaki dynamics with another process for which this problem is easier. This method uses Girsanov's theorem and specific information on the operator $A$, and I do not know how to replicate it in the abstract setting considered here. Moreover, unlike our method, it does not work when $t' = 0$. However, for discrete spins, the quantity that would play the role of $\int_Y{\tilde{f}\log \tilde{f} dy}$ is non-positive, which makes her method very convenient when applied to particle systems such as exclusion processes. It is not clear whether the two-scale approach can be successfully applied to the study of discrete systems.
\end{rmq}

\subsection{Kawasaki dynamics}

We shall now present the application of the two previous theorems to Kawasaki dynamics. 
We consider a one dimensional $N$-periodic lattice system with continuous spin variables. The law of each variable is given by a Ginzburg-Landau potential 
$\psi : \mathbb{R} \rightarrow \mathbb{R}$, which we shall assume to be of the form 
\begin{equation} \label{24}
\psi(x) = \frac{1}{2}x^2 + \delta \psi(x), \hspace{0.5cm} ||\delta \psi||_{C^2(\mathbb{R})} < \infty.
\end{equation}
We shall also force the mean spin to take a given value $m \in \mathbb{R}$. That is, the random vector $x = (x_1,..,x_N)$ will take its values in the $(N-1)$-dimensional hyperplane with mean $m \in \mathbb{R}$ : 
$$X_{N,m} := \left\{(x_1,..,x_N) \in \mathbb{R}^N; \hspace{0.1cm} \frac{1}{N}\sum x_i = m\right\}$$
equipped with the $\ell^2$ inner product,
$$\langle x, \tilde{x} \rangle_{X_{N,m}} := \sum x_i \tilde{x}_i.$$
We shall consider the canonical ensemble $\mu_{N,m}$, which is the distribution of the random variables $x_1,..,x_N$ conditioned on the event that their mean value is given by $m \in \mathbb{R}$. Its density with respect to the Lebesgue measure on $X_{N,m}$ is given by
\begin{equation} \label{25}
\mu_{N,m}(dx) = \frac{1}{Z}1_{\sum x_i = Nm}\exp\left(-\underset{i=1}{\stackrel{N}{\sum}} \psi(x_i) \right).
\end{equation}
The logarithmic density $H$ is evidently given by $H(x) = \underset{i=1}{\stackrel{N}{\sum}} \psi(x_i) + \log Z$.

We shall now introduce the macroscopic state necessary to apply the abstract results of the previous section. 
We first divide the N spins into M blocks. To fix ideas, we shall assume that all these blocks have the same size $K$, 
such that $N = KM$. This assumption is not necessary (all that will be needed is that the sizes of all the blocks are of same order) 
but it will make things a lot clearer. See [GOVW Remark 30] for a full explanation about this. 
We will now define the macroscopic variables as the mean of each block. Therefore they form a set of $M$ real numbers that still have mean $m$. 
The associated macroscopic space is thus
$$Y_{M,m} := \left\{(y_1,..,y_M) \in \mathbb{R}^M;\hspace{0.1cm} \frac{1}{M}\sum y_i = m\right\}$$
which we endow with the $\textsl{L}^2$ inner product 
$$\langle y, \tilde{y} \rangle_Y := \frac{1}{M}\sum y_j \tilde{y}_j.$$
Then the projection operator $P_{N,K} : X_{N,m} \rightarrow Y_{M,m}$ that associates to a given microscopic profile its macroscopic profile is given by
$$P_{N,K}(x_1,..,x_N) = (y_1,..,y_M); \hspace{1cm} y_j = \frac{1}{K}\underset{i = (j-1)K + 1}{\stackrel{jK}{\sum}} x_i,$$
and it is easy to check that $PNP^t = id_Y$. We can explicitly compute the coarse-grained Hamiltonian $\bar{H}$:
$$\bar{H}(y) = \frac{1}{M}\underset{j=1}{\stackrel{M}{\sum}}\psi_K(y_i) + \frac{1}{N}\log \bar{Z}$$
where 
\begin{equation} \label{22}
\psi_K(m) = -\frac{1}{K}\log \left(\int_{X_{K,m}}\exp(-\underset{i=1}{\stackrel{K}{\sum}}\psi(x_i))dx\right)
\end{equation}
and $\bar{Z}$ is the normalization constant. The gradient and Hessian of $\bar{H}$ are then given by
\begin{equation} \label{23}
(\nabla_Y \bar{H}(y))_Y = \psi_K'(y_i); \hspace{1cm} (\operatorname{Hess}_Y \bar{H})_{ij} = \psi_K''(y_i)\delta_{ij}.
\end{equation}

As a consequence of the principle of equivalence of ensembles (quantified through a local version of the Cram\'er theorem), the following proposition explains the behavior of $\psi_K$ when $K$ is large. It was proven in the Appendix of [GOVW]. 

\begin{prop} \label{203} If $\psi$ satisfies (\ref{24}) and $\psi_K$ is defined by (\ref{22}), then 
$$\psi_K \underset{K \uparrow \infty}{\rightarrow} \varphi \hspace{1cm} \text{in the uniform } \textsl{C}^2 \text{ topology},$$
where $\varphi$ is the Cram\'er transform of $\psi$, defined by
\begin{equation} \label{27}
\varphi(m) = \underset{\sigma \in \mathbb{R}}{\sup}\left(\sigma m - \log \int{\exp(\sigma x - \psi(x))dx}\right).
\end{equation}
\end{prop}

Using this proposition, the strict convexity of $\varphi$ and the expression of the Hessian (\ref{23}), the following lemma is easily deduced : 
\begin{lem} [Convexity of the coarse-grained Hamiltonian]
There exists $K_0 < \infty$ and $\lambda > 0$ depending only on $\psi$ such that, for any $K \geq K_0$,  $$\left\langle \tilde{y}, \operatorname{Hess}\bar{H}(y)\tilde{y}\right\rangle_Y \geq \lambda \langle \tilde{y}, \tilde{y} \rangle_Y.$$
\end{lem}

This lemma, among others, allowed to apply the abstract criterion for logarithmic Sobolev inequalities to the present setting, and obtain

\begin{thm} Let $\psi$ satisfy (\ref{24}) and let $\mu_{N,m}$ be defined by (\ref{25}). Then there exists $\rho > 0$ such that for any $N \in \mathbb{N}$ and $m \in \mathbb{R}$, $\mu_{N,m}$ satisfies LSI($\rho$).
\end{thm}

This result was recently extended in [MO] to the case where $\psi$ is a bounded perturbation of a uniformly convex function (rather than strictly quadratic), using a technique of iterated coarse-graining. 

We shall now present the Kawasaki dynamics for such a system of spins. We (arbitrarily) set the mean $m$ in the setting just explained to be $0$, and we consider a dynamics of the form described by (\ref{8}), with the matrix $A = (A_{ij})$ defined by 
\begin{equation} 
A_{ij} = N^2(-\delta_{i,j-1} + 2\delta_{i,j} - \delta_{i,j+1})
\end{equation}
We also identify the space $X_{N,0}$ with the space $\bar{X}$ of piecewise constant functions on $\mathbb{T} = \mathbb{R}/\mathbb{Z}$ :
$$\bar{X} = \left\{\bar{x} : \mathbb{T} \rightarrow \mathbb{R}; \hspace{1cm} \bar{x} \text{ is constant on } \left(\frac{j-1}{N},\frac{j}{N}\right], \hspace{0.5cm} j = 1,..,N\right\}$$
by associating to the vector $x \in X_{N,0}$ the function $\bar{x} \in \bar{X}$ such that
$$\bar{x}(\theta) = x_j, \hspace{2cm} \theta \in \left(\frac{j-1}{N},\frac{j}{N}\right].$$

To obtain the final hydrodynamic limit, we must embed all of these spaces $\bar{X}_N$ in a common functional space. We consider the space of functions $f : \mathbb{T} \rightarrow \mathbb{R}$ of locally integrable functions of mean zero, which we equip of the following norm: 
\begin{equation} ||f||_{H^{-1}}^2 = \int_{\mathbb{T}}{w^2(\theta)d\theta}; \hspace{1cm} w' = f, \hspace{0.5cm} \int_{\mathbb{T}}{w(\theta)d\theta} = 0.
\end{equation}
Then the closure of all the spaces $\bar{X}_N$ for this norm is the usual Sobolev space $\textsl{H}^{-1}(\mathbb{T})$. We can now formulate the following theorem on the hydrodynamic limit of the Kawasaki dynamics.

\begin{thm}
\label{71}
Assume that $\psi$ satisfies (\ref{24}). Let $f_N = f_N(t,x)$ be a time-dependent probability density on $(X_{N,0},\mu_{N,0})$ solving 
$$\frac{\partial}{\partial t}(f \mu_{N,0}) = \nabla \cdot (A \nabla f \mu_{N,0})$$
where $f_N(0,\cdot) = f_{0,N}(\cdot)$ satisfies
\begin{equation}
\int{f_{0,N}(x)\log f_{0,N}(x) \mu_{N,0}(dx)} \leq C N
\end{equation}
for some constant $C > 0$. Assume that 
\begin{equation}
\underset{N \uparrow \infty}{\lim} \int{||\bar{x} - \zeta_0||_{H^{-1}}^2 f_{0,N}(x)\mu_{N,0}(dx)} = 0
\end{equation}
for some $\zeta_0 \in \textsl{L}^2(\mathbb{T})$ which has mean zero. Then for any $T > 0$ we have 
\begin{equation}
\underset{N \uparrow \infty}{\lim} \hspace{2mm} \underset{0 \leq t \leq T}{\sup} \int{||\bar{x} - \zeta(t,\cdot)||_{H^{-1}}^2 f_{N}(t,x)\mu_{N,0}(dx)} = 0,
\end{equation}
where $\zeta$ is the unique weak solution of the nonlinear parabolic equation
\begin{equation} \label{28}
\frac{\partial \zeta}{\partial t} = \frac{\partial^2}{\partial \theta^2}\varphi'(\zeta)
\end{equation}
with initial condition $\zeta(0,\cdot) = \zeta_0(\cdot)$, where $\varphi$ is defined as in (\ref{27}).
\end{thm}

In this theorem, a weak solution of (\ref{28}) is defined in the following way: 

\begin{defn}
We will call $\zeta = \zeta(t,\theta)$ a weak solution of (\ref{28})  on $[0,T] \times \mathbb{T}$ if
\begin{equation} \zeta \in L_t^{\infty}(L_{\theta}^2), \hspace{1cm} \frac{\partial\zeta}{\partial t} \in L_t^2(\textsl{H}_{\theta}^{-1}), \hspace{1cm} \varphi'(\zeta) \in L_t^2(L_{\theta}^2),
\end{equation}
and
\begin{equation} \left\langle \xi,\frac{\partial \zeta}{\partial t}\right\rangle_{H^{-1}} = -\int_{\mathbb{T}^1}{\xi \varphi'(\zeta) d\theta} \hspace{1cm} \text{ for all } \xi \in L^2\text{, for a.e. } t \in [0,T]
\end{equation}
\end{defn}

One of the main steps of the proof, which will also be used in this paper, is the convergence of $\bar{\eta}$ to $\zeta$: 
\begin{prop}
Let $\bar{\eta}_0^{\ell} \in \bar{Y}_{\ell}$ be a step function approximation of  $\zeta_0$, $\eta_0^{\ell}$ the vector of Y associated with it, and $\eta_{\ell}$ the solution of (\ref{9}) with initial condition $\eta_0^{\ell}$. Then the step functions $\bar{\eta}_{\ell}$ converge strongly in $\textsl{L}_t^{\infty}(\textsl{H}_{\theta}^{-1})$ to the unique weak solution of (\ref{28}) with initial condition $\zeta_0$.
\end{prop}

We now state the result obtained when applying the previous abstract theorem to this setting.

\begin{thm} [Convergence of the entropy for Kawasaki dynamics] \label{30} 
Under the same assumptions as Theorem \textbf{\ref{71}}, the relative entropy with respect to the local Gibbs state goes to zeo, in a time-integrated sense : 
\begin{equation} \label{ent1}
\int_0^T{\int_{X_N}{\Phi\left(\frac{f_N(t,x)}{G_N(t,x)}\right)G_N(t,x) \mu_N(dx)}dt},
\end{equation}
where $G_N(t,\cdot)$ is the local Gibbs state given by $\eta_N(t)$. 
As a consequence, we have convergence of the microscopic entropy to the hydrodynamic entropy, in a time-integrated sense : 
\begin{equation} \label{ent2}
\int_0^T{\left|\frac{1}{N}\int{\Phi(f_N(t,x))\mu_N(dx)} - \left(\int_{\mathbb{T}}{\varphi(\zeta(\theta,t))d\theta} - \varphi\left(\int_{\mathbb{T}}{\zeta(t,\theta) d\theta}\right)\right)\right|dt} \underset{N \rightarrow \infty}{\rightarrow} 0.
\end{equation}
\end{thm}

Moreover, in this setting, we will be able to get a pointwise convergence of the entropy, as long as we stay away from the origin. It will follow from the time-integrated convergence and the fact that the entropy is decreasing in time.

\begin{thm}[Pointwise convergence of the relative entropy] \label{77}
Assume that $\zeta$ is continuous in both variables. Let $0 < \epsilon < T$. Then 
$$\frac{1}{N}\int{\Phi(f_N(t,x))\mu_N(dx)} \underset{N \rightarrow \infty}{\longrightarrow} \int_{\mathbb{T}}{\varphi(\zeta(\theta,t))d\theta} - \varphi\left(\int_{\mathbb{T}}{\zeta(t,\theta) d\theta}\right)$$
uniformly on $[\epsilon, T]$.
\end{thm}

\begin{rmq} This convergence will in general not hold true at initial time, since no relation is assumed between the initial microscopic entropy and the initial hydrodynamic entropy. However, if it does hold true initially, then it will hold true uniformly on $[0,T]$ for any $T > 0$. (This is the main outcome of Yau's entropy method.)
\end{rmq}

Since we do not necessarily assume our initial data to be smooth, $\zeta$ is not in general smooth at $t = 0$. However, as long as $\zeta_0$ lies in $L^2(\mathbb{T})$, $\zeta$ will satisfy the smoothness assumptions of Theorem \ref{77}:

\begin{prop} \label{reg_hydro}
Assume $\varphi$ is a $C^3$ function, with $\varphi'' \geq \lambda > 0$, $||\varphi''||_{\infty} < \infty$ and $||\varphi^{(3)}/\varphi''||_{\infty} < \infty$. Let $\zeta$ the weak solution of 
$$\frac{\partial \zeta}{\partial t} = \frac{\partial^2}{\partial \theta^2}\varphi'(\zeta)$$
with initial data $\zeta_0 \in L^2(\mathbb{T})$. Then, for any $\epsilon > 0$, $\zeta$ lies in $C^{1,2}([\epsilon, T] \times \mathbb{T})$, and $\frac{\partial}{\partial t}\varphi'(\zeta)$ and $\frac{\partial^2}{\partial \theta^2}\varphi'(\zeta)$ are uniformly continuous.
\end{prop}

This result is well-known in the PDE community, we give a proof for the sake of completeness.

\section{Proof of Theorem \ref{4}}

Let us first state some properties of the Local Gibbs state, which we shall use for the proof.

\begin{prop}[Study of the local Gibbs state] \label{11}
(i) At the macroscopic scale the density $\bar{G}\bar{\mu}$ is given by 
$$\bar{G}(y)\bar{\mu}(dy) = \frac{1}{\bar{Z}}\exp\left(N(\nabla \bar{H}(\eta) \cdot y - \bar{H}(y))\right)dy.$$

(ii) At the macroscopic scale, we have the following bound of the Wasserstein distance between the local Gibbs state given by $\bar{G}(\cdot)d\bar{\mu}$ and $\delta_{\eta}$ : 
$$\int{|y - \eta|_Y^2\bar{G}(y)\bar{\mu}(dy)} \leq \frac{M}{\lambda N}.$$

(iii) The free energy associated with $G^{\eta}$ is close to the energy associated with $\eta$, with the explicit bound 
\begin{align}
&\left|\frac{1}{N}\int{\Phi(G^{\eta})d\mu} - \bar{H}(\eta)\right| \notag \\
&\leq \frac{(M-1)}{2N}\max\left(\left|\log\left(\frac{\Gamma(Y,|\cdot|_Y)^{2/(M-1)}}{\Lambda N}\right)\right|,\left|\log\left(\frac{\Gamma(Y,|\cdot|_Y)^{2/(M-1)}}{\lambda N}\right)\right|\right) \notag \\
& + \sqrt{\frac{M}{\lambda N}} |\nabla \bar{H}(\eta)|. \notag
\end{align}
\end{prop}

\begin{rmq} We can use the same techniques as in [GOVW] to pass from macroscopic to microscopic scale, and deduce from (ii) a bound on the penalized Wasserstein distance 

$\frac{1}{N}W_{A^{-1}}(Gd\mu, \delta_{NP^t\eta(t)}) = \frac{1}{N}\int{\langle A^{-1}(x - NP^t\eta), x - NP^t\eta \rangle \hspace{0.1cm} G(x)\mu(dx)}$.
\end{rmq}

Part (iii) will be how we deduce the convergence of the entropy from the local Gibbs behavior. The proof of this proposition will hinge on the following lemma, which tells us that among all the probability measures on $\mathbb{R}^M$ of the form $\exp(-f(x))dx$, with $\text{Hess} f \geq \lambda Id$ and where $f$ reaches its minimum in 0, the one with the highest second moment is the centered Gaussian of covariance matrix $\lambda \hspace{0.1cm} id$.

\begin{lem} \label{1}
If $f : \mathbb{R}^M \rightarrow \mathbb{R}$ is $C^2$ and uniformly convex, with $\operatorname{Hess } f \geq \lambda \text{Id}$, $\lambda > 0$, and $\min f = f(0)$, then 
$$\int{|x|^2e^{-f(x)}dx} \leq \frac{M}{\lambda}\int{e^{-f(x)}dx}.$$
\end{lem}

The following proof of this lemma was pointed out to us by S.R.S. Varadhan.

\begin{proof} 
Fix $x \in \mathbb{R}^M$. The function $g(t) = f(tx) - \frac{\lambda}{2}||tx||_2^2$ is convex, and reaches its minimum for $t = 0$. Therefore, we have $g'(1) \geq g'(0) = 0$. Since $g'(1) = \langle x, \nabla f (x)\rangle - \lambda |x|^2$, we obtain 
$$\lambda \int{|x|^2e^{-f(x)}dx} \leq \int{\langle x, \nabla f (x)\rangle e^{-f(x)}dx}$$
and, by integration by parts, the term on the right-hand side is equal to $M \int{e^{-f(x)}dx}$, which concludes the proof.

\end{proof}

\noindent This lemma will allow us to bound the Wasserstein distance in (ii).

\begin{proof} [Proof of Proposition \ref{11}]

\noindent The proof of (i) is trivial, since we have constructed our local Gibbs state such that $G(x)$ actually only depends on $Px$. For all $t$, since $y \rightarrow \psi_t(y) = \bar{H}(y) - \nabla \bar{H}(\eta(t)) \cdot y$ is uniformly convex, with its Hessian bound below by $\lambda Id$, and reaches its minimum for $y = \eta(t)$, applying lemma \ref{1} with $f = N\psi_t$, after translating by $\eta(t)$, we obtain 
\begin{align}
\int{|y - \eta|_Y^2\bar{G}(y)\bar{\mu}(dy)} &= \frac{1}{\bar{Z}}\int{|y - \eta|_Y^2 e^{-N\psi_t(y)}dy} \notag \\
&\leq \frac{M \bar{Z}}{\lambda N \bar{Z}} 
\end{align}
which yields (ii).

\noindent For (iii), we have 
\begin{align} \label{200}
&\left|\frac{1}{N}\int{\Phi(G^{\eta})d\mu} - \bar{H}(\eta)\right| = \left|\frac{1}{N}\int{(N\nabla \bar{H}(\eta) \cdot y - \log \bar{Z})\bar{G}^{\eta}(y)\bar{\mu}(dy)} - \bar{H}(\eta)\right| \notag \\
&\hspace{0.1cm} \leq \left|-\frac{1}{N}\log \bar{Z} - \bar{H}(\eta) + \nabla \bar{H}(\eta)\cdot \eta \right| + \left|\frac{1}{N}\int{\nabla \bar{H}(\eta)\cdot (y - \eta)\bar{G}(y)\bar{\mu}(dy)}\right|
\end{align}

Since $\lambda \operatorname{Id} \leq \operatorname{Hess} \bar{H} \leq \Lambda \operatorname{Id}$, we have the bounds
\begin{align} \frac{\lambda}{2}|y - \eta|^2 &\leq -\nabla \bar{H}(\eta) \cdot (y - \eta) - \bar{H}(\eta) + \bar{H}(y) \notag \\
&\leq \frac{\Lambda}{2}|y - \eta|^2 \notag
\end{align}
for all $y \in Y$. We now multiply by $N$ and integrate. We obtain the upper bound 
\begin{align}
-\frac{1}{N}\log \bar{Z} - &\bar{H}(\eta) + \nabla \bar{H}(\eta) \cdot \eta \notag \\
&= -\frac{1}{N}\log\int{\exp(N(\nabla \bar{H}(\eta)\cdot y - \bar{H}(y)))dy} - \bar{H}(\eta) + \nabla \bar{H}(\eta) \cdot \eta \notag \\
&= -\frac{1}{N}\log\int{\exp(N(\nabla \bar{H}(\eta)\cdot (y - \eta) - (\bar{H}(y)) - \bar{H}(\eta)))dy} \notag \\
&\leq -\frac{1}{N}\log \int{\exp(-\frac{\Lambda N}{2}|y - \eta|_Y^2)dy} \notag \\
&= -\frac{1}{N}\log\left(\left(\frac{1}{\sqrt{\Lambda N}}\right)^{M-1} \int{\exp(-|y|_Y^2/2)dy}\right) \notag \\
&= -\frac{M-1}{2N}\log\left(\frac{\Gamma(Y,|\cdot |_Y)^{2/(M-1)}}{\Lambda N}\right) \notag
\end{align}
In the same way, we obtain a lower bound, so that when we take the absolute value we get the bound

$$\left|\frac{1}{N}\log \bar{Z} - \bar{H}(\eta) + \nabla \bar{H}(\eta) \cdot \eta \right|$$ 
\begin{equation}\leq \frac{M-1}{2N}\max\left(\left|\log\left(\frac{\Gamma(Y,|\cdot |_Y)^{2/(M-1)}}{\lambda N}\right)\right|\left|\log\left(\frac{\Gamma(Y,|\cdot |_Y)^{2/(M-1)}}{\Lambda N}\right)\right|\right).
\end{equation}
Combined with (\ref{200}), this means the final element we need is a bound on $\int{\nabla \bar{H}(\eta) \cdot (y - \eta) \bar{G}^{\eta}(y)\bar{\mu}(dy)}$, 
and simply using the Cauchy-Schwartz inequality and (ii) gives us the desired result.

\end{proof}

We shall now use these properties of the local Gibbs State to prove Theorem \textbf{\ref{4}}.

\begin{proof} [Proof of Theorem \ref{4}]

We shall divide this proof into three steps : first we shall reduce the problem to the study of the time-integrated relative entropy between the macroscopic state $\bar{f}(t, \cdot)\bar{\mu}$ and the local Gibbs state $\bar{G}^{\eta(t)}\bar{\mu}$, 
then we shall use the HWI interpolation inequality to show that this relative entropy goes to $0$. 
Finally, we shall reintroduce the microscopic terms to get the full bound (a). Then (b) shall follow, since we have already proved in Proposition \ref{1} that the free energy associated with $G^{\eta}$ is asymptotically close to $\bar{H}(\eta(t))$.

\vspace{0.2cm}

\textbf{Step 1 :} Let us consider
\begin{align} \label{50}
\frac{1}{N}H_N(t) &= \int{\Phi\left(\frac{f(t,x)}{G(t,x)}\right)G(t,x) \mu_N(dx)} \notag \\
&= \frac{1}{N}\operatorname{Ent}_{\mu}(f) - \frac{1}{N}\int{f \log(G) \mu(dx)} \notag \\
&= \frac{1}{N}\int{\operatorname{Ent}_{\mu(dx|y)}(f)\bar{\mu}(dy)} + \frac{1}{N}\operatorname{Ent}_{\bar{\mu}}(\bar{f}) - \frac{1}{N}\int{f \log(G) \mu(dx)} \notag \\
&= \frac{1}{N}\int{\operatorname{Ent}_{\mu(dx|y)}(f)\bar{\mu}(dy)} + \frac{1}{N}\operatorname{Ent}_{\bar{G}\bar{\mu}}(\frac{\bar{f}}{\bar{G}}).
\end{align}
where the last equality is obtained because $G(x)$ only depends on the macroscopic state $Px$. Therefore, to reduce the problem to the study of the macroscopic entropy, we just have to produce an appropriate bound on 

$\frac{1}{N}\int{Ent_{\mu(dx|y)}(f)\bar{\mu}(dy)}.$ Since $\mu(dx|y)$ satisfies the condition LSI($\rho$) by assumption (ii), we have
\begin{align} \frac{1}{N}\int_Y{\operatorname{Ent}_{\mu(dx|y)}}&(f)\bar{\mu}(dy) \notag \\
& \leq \frac{1}{N \rho} \int_X{\frac{|(id_X - P^tNP)\nabla f(x)|^2}{2f(x)}\mu(dx)} \notag \\
& \leq \frac{\gamma}{2NM^2\rho}\int{\frac{\langle \nabla f(x), A \nabla f(x) \rangle}{f(x)}\mu(dx)},
\end{align}
where the last inequality is due to hypothesis (vi). By integrating by parts (\ref{8}), as was done in [GOVW, Proposition 24], we deduce that 
$$\int{f(T,x)\log f(T,x)\mu(dx)} + \int_0^T{\left(\int{\frac{\nabla f \cdot A\nabla f}{f}(t,x)\mu(dx)}\right)dt}$$
\begin{equation}
 = \int{f(0,x)\log f(0,x)\mu(dx)}.
\end{equation}
Since, by assumption (vii), $\int{\Phi(f(0,x))\mu(dx)} \leq C_1N$ and the (mathematical) entropy is non-negative, this tells us that
\begin{equation} \label{19}
\int_0^T{\frac{1}{N}\int{\operatorname{Ent}_{\mu(dx|y)}(f(t,\cdot))\bar{\mu}(dy)}dt} \leq \frac{\gamma C_1}{2M^2\rho},
\end{equation}
which concludes this first step of the proof.

\vspace{0.2cm}

\textbf{Step 2 :} We shall now study the macroscopic relative entropy. This is where we shall use he HWI inequality (Theorem \ref{500}), applied at the macroscopic scale with reference measure $\bar{\mu}$. Indeed, since $\bar{G}\bar{\mu}$ is log-concave, we have for all $T > 0$
\begin{align} \label{17}
\int_0^T{\frac{1}{N}\operatorname{Ent}_{\bar{G}(t,\cdot)\bar{\mu}}}&\left(\frac{\bar{f}(t,\cdot)}{\bar{G}(t,\cdot)}\right)dt \leq \int_0^T{\frac{1}{N}W_2(\bar{f}(t,\cdot)\bar{\mu},\bar{G}(t,\cdot)\bar{\mu}) \hspace{1mm} \sqrt{I_{\bar{G}(t,\cdot)\bar{\mu}}(\bar{f}(t,\cdot)\bar{\mu})}dt} \notag \\
&\leq \sqrt{\int_0^T{W_2(\bar{f}(t,\cdot)\bar{\mu},\bar{G}(t,\cdot)\bar{\mu})^2dt}} \hspace{1mm} \sqrt{\frac{1}{N^2}\int_0^T{I_{\bar{G}(t,\cdot)\bar{\mu}}(\bar{f}(t,\cdot)\bar{\mu})dt}}
\end{align}
We already know that both $\bar{f}\bar{\mu}$ and $\bar{G}\bar{\mu}$ are asymptotically close to $\delta_{\eta}$. By the triangle inequality for the Wasserstein distance, 
\begin{equation} \label{13}
W_2(\bar{f}(t,\cdot)\bar{\mu},\bar{G}(t,\cdot)\bar{\mu})^2 \leq 2 \hspace{1mm} W_2(\bar{f}(t,\cdot)\bar{\mu},\delta_{\eta(t)})^2 +2 \hspace{1mm}W_2(\delta_{\eta(t)},\bar{G}(t,\cdot)\bar{\mu})^2
\end{equation}
Theorem \textbf{\ref{10}} states that
$$\int_0^T{W_2(\bar{f}(t,\cdot)\bar{\mu},\delta_{\eta(t)})^2dt} = \int_0^T{\int{|y - \eta(t)|_Y^2\bar{f}(t,y)\bar{\mu}(dy)}dt}$$
\begin{equation} \label{12}
\leq \frac{2}{\lambda}\Xi(T,M,N)
\end{equation}
where $\Xi$ was defined  in Theorem \textbf{\ref{10}}. Moreover, part (ii) of Proposition \textbf{\ref{11}} tells us that
\begin{equation} \label{16}
\int_0^T{\int{|y-\eta(t)|_Y^2\bar{G}(t,y)\bar{\mu}(dy)}dt} \leq T\frac{M}{\lambda N}
\end{equation}
so we have a bound on the time-integral of the Wasserstein distance that, under suitable assumptions, will go to $0$.

We must now produce a bound on the macroscopic Fisher information. We have
\begin{align} \label{99}
\frac{1}{N^2}I_{\bar{G}\bar{\mu}}\left(\frac{\bar{f}}{\bar{G}}\right) &= \frac{1}{N^2}\int{\frac{|\nabla (\bar{f}/\bar{G})|^2}{\bar{f}/\bar{G}}\bar{G}d\bar{\mu}} \notag \\
&= \frac{1}{N^2}\int{\frac{|(\nabla \bar{f})/\bar{G} - \bar{f}\nabla \bar{G}/\bar{G}^2|^2}{\bar{f}}\bar{G}^2d\bar{\mu}} \notag \\
&\leq \frac{2}{N^2}\int{\frac{|\nabla \bar{f}|^2}{\bar{f}}d\bar{\mu}} + \frac{2}{N^2}\int{\frac{|\nabla \bar{G}|^2}{\bar{G}^2}\bar{f}d\bar{\mu}} \notag \\
&= \frac{2}{N^2}\int{\frac{|\nabla \bar{f}|^2}{\bar{f}}d\bar{\mu}} + 2|\nabla \bar{H}(\eta)|^2
\end{align}
Now, since we have a lower bound $\tau$ on the spectral values of $A$, $1/\tau$ is an upper bound on the spectral values of $A^{-1}$. Since for any $y \in Y$
\begin{align}
\langle \bar{A}^{-1} y, y \rangle_Y &= \langle PA^{-1}NP^t y, y \rangle_Y \notag \\
&= \frac{1}{N} \langle A^{-1}NP^t y, NP^t y\rangle_X \notag \\
&\leq \frac{1}{N\tau} \langle NP^t y, NP^t y\rangle_X \notag \\
&= \frac{1}{\tau}|y|_Y^2, \notag
\end{align}
so $1/\tau$ is an upper bound on the spectral values of $\bar{A}^{-1}$, and thus $\tau$ is also a lower bound on the spectral values of $\bar{A}$. Therefore
\begin{align} \label{14}
\int_0^T{|\nabla \bar{H}(\eta(t))|^2dt} &\leq \frac{1}{\tau}\int_0^T{\langle \bar{A} \nabla \bar{H}(\eta(t)), \nabla \bar{H}(\eta(t)) \rangle dt} \notag \\
&= -\frac{1}{\tau}\int_0^T{\langle \frac{d\eta}{dt}(t), \nabla \bar{H}(\eta(t)) \rangle dt} \notag \\
&= \frac{1}{\tau}(\bar{H}(\eta(0)) - \bar{H}(\eta(T))) \notag \\
&\leq \frac{C_2 + \beta}{\tau}.
\end{align}

To obtain a bound on $\int_0^T{\int{\frac{|\nabla \bar{f}|^2}{\bar{f}}d\bar{\mu}}dt}$, we shall use the following proposition, that was proved in [GOVW, Proposition 20].

\begin{prop}
Assume that $\kappa$ as given by (\ref{5}) is finite and that for all $y \in Y$, $\mu(dx|y)$ satisfies LSI($\rho$). Then, for any positive, $\textsl{C}^1$ function on $X$ one has, for any $y \in Y$ and $s \in (0,1)$
\begin{align} \frac{1}{N}\frac{|\nabla_Y \bar{f}(y)|_Y^2}{\bar{f}(y)} &\leq \frac{1}{1 - s}\left(\frac{\kappa^2}{\rho^2}\right)\int{\frac{|(id_X - PNP^t)\nabla f(x)|^2}{f(x)}\mu(dx|y)} \notag \\
&+ \frac{1}{s}\int{\frac{|PNP^t\nabla f(x)|^2}{f(x)}\mu(dx|y)}
\end{align}
\end{prop}

Applying this bound to the density $f$ with $s = \frac{\rho^2}{\kappa^2 + \rho^2}$ gives us the bound
\begin{align} \label{15} \frac{1}{N^2}&\int_0^T{\int{\frac{|\nabla \bar{f}|^2}{\bar{f}}d\bar{\mu}}dt} \leq \frac{\kappa^2 + \rho^2}{\rho^2}\frac{1}{N}\int_0^T{\int{\frac{|\nabla f(x)|^2}{f(x)}\mu(dx)}dt} \notag \\
&\leq \frac{\kappa^2 + \rho^2}{\rho^2}\frac{1}{N}\frac{1}{\tau}\int_0^T{\int{\frac{\nabla f \cdot A\nabla f}{f}d\mu}dt} \notag \\
&\leq \frac{C_1(\kappa^2 + \rho^2)}{\tau \rho^2}
\end{align}
where the final inequality was already proved in step 1.

\noindent Combining (\ref{17}), (\ref{13}), (\ref{12}), (\ref{16}), (\ref{99}), (\ref{14}) and (\ref{15}) gives us the bound
\begin{align} \label{18}
&\int_0^T{\frac{1}{N}Ent_{\bar{G}(t,\cdot)\bar{\mu}}}(\frac{\bar{f}(t,\cdot)}{\bar{G}(t,\cdot)})dt \notag \\
&\leq \sqrt{\frac{2TM}{N} + \frac{4}{\lambda}\Xi(T,M,N)} \notag \\
& \hspace{2cm} \times \sqrt{2\frac{C_2 + \beta}{N^2\tau} + 2\frac{C_1(\kappa^2 + \rho^2)}{\tau \rho^2}}.
\end{align}
This concludes Step 2.

\vspace{0.2cm}

\noindent \textbf{Step 3 :} Recombining (\ref{19}) and (\ref{18}) gives us the full bound
\begin{align}&\int_0^T{\frac{1}{N}\int{\Phi\left(\frac{f(t,x)}{G(t,x)}\right)G(t,x) \mu_N(dx)}dt} \notag \\
& \leq \sqrt{\frac{2TM}{N} + \frac{4}{\lambda}\Xi(T,M,N)}\times \sqrt{2\frac{C_2 + \beta}{N^2\tau} + 2\frac{C_1(\kappa^2 + \rho^2)}{\tau \rho^2}} + \frac{\gamma C_1}{2M^2\rho}.
\end{align}
To prove (b), we have
\begin{align} \frac{1}{N}&\int{\Phi\left(\frac{f(t,x)}{G(t,x)}\right)G(t,x) \mu_N(dx)} \notag \\
&= \left|\frac{1}{N}\int{f(t,x)\log f(t,x) \mu(dx)} - \frac{1}{N}\int{f(t,x) \log G(t,x) \mu(dx)}\right| \notag \\
&= \left|\frac{1}{N}\int{f(t,x)\log f(t,x) \mu(dx)} - \frac{1}{N}\int{\bar{f}(t,y) \log \bar{G}(t,y) \bar{\mu}(dy)}\right| \notag  \end{align}
and thus
\begin{align} \label{20}
&\int_0^T{\left|\frac{1}{N}\int{f(t,x)\log f(t,x) \mu(dx)} - \bar{H}(\eta(t))\right|dt} \notag \\
&\leq \int_0^T{\left|\frac{1}{N}\int{f(t,x)\log f(t,x) \mu(dx)} - \frac{1}{N}\int{f(t,x) \log G(t,x) \mu(dx)}\right|dt} \notag \\
& \hspace{1cm} +\int_0^T{ \left|\frac{1}{N}\int{f(t,x) \log G(t,x) \mu(dx)} - \bar{H}(\eta(t))\right|dt} \notag \\
&\leq \int_0^T{\frac{1}{N}\int{\Phi\left(\frac{f(t,x)}{G(t,x)}\right)G(t,x) \mu_N(dx)}dt} \notag \\
&\hspace{0.5cm}+ \int_0^T{\left|\frac{1}{N}\int{\bar{f}(t,y) \log \bar{G}(t,y) \bar{\mu}(dy)} - \bar{H}(\eta(t))\right|dt} \notag \\
&= \int_0^T{\frac{1}{N}\int{\Phi\left(\frac{f(t,x)}{G(t,x)}\right)G(t,x) \mu_N(dx)}dt} \notag \\
& \hspace{0.5cm} + \int_0^T{\left|\frac{1}{N}\int{(\bar{f}(t,y) -\bar{G}(t,y))\log \bar{G}(t,y)\bar{\mu}(dy)}\right| dt} \notag \\
& \hspace{1cm} + \int_0^T{\left|\frac{1}{N}\int{\Phi(G(t,\cdot))d\mu_N} - \bar{H}(\eta(t))\right|dt}. 
\end{align}

\noindent This leaves us with three quantities to bound. The first one is exactly the quantity that is bounded by (a). The third quantity can be bounded using part (iii) of Proposition \textbf{\ref{11}} :

\begin{align}
\int_0^T&{\left|\frac{1}{N}\int{\Phi(G^{\eta})d\mu_N} - \bar{H}(\eta)\right|dt} \notag \\
&\leq \frac{T(M-1)}{2N}\max\left(\left|\log\left(\frac{\Gamma(Y,|\cdot|_Y)^{2/(M-1)}}{\Lambda N}\right)\right|\left|\log\left(\frac{\Gamma(Y,|\cdot|_Y)^{2/(M-1)}}{\lambda N}\right)\right|\right) \notag \\
& \hspace{1cm} +\int_0^T{\sqrt{\frac{M}{\lambda N}} |\nabla \bar{H}(\eta)|dt} \notag \\
&\leq \frac{T(M-1)}{2N}\max\left(\left|\log\left(\frac{\Gamma(Y,|\cdot|_Y)^{2/(M-1)}}{\Lambda N}\right)\right|\left|\log\left(\frac{\Gamma(Y,|\cdot|_Y)^{2/(M-1)}}{\lambda N}\right)\right|\right) \notag \\
& \hspace{1cm} +\sqrt{\frac{TM}{\lambda N}\frac{C_2 + \beta}{\tau}}
\end{align}
where the final inequality is due to (\ref{14}). To conclude, we just have to bound the second quantity in the right-hand side of (\ref{20}). This will be possible because, as $\log \bar{G}$ is slowly varying (it is an affine function), since $\bar{f}\mu$ and $\bar{G}\mu$ are close for the second Wasserstein distance, when integrating against $\log \bar{G}$ they act the same way.

\begin{align}
&\int_0^T{\left|\frac{1}{N}\int{(\bar{f}(t,y) -\bar{G}(t,y))\log \bar{G}(t,y)\bar{\mu}(dy)}\right| dt} \notag \\
&\hspace{1cm} = \int_0^T{\left|\frac{1}{N}\int{(\bar{f}(t,y) -\bar{G}(t,y))(N\nabla \bar{H}(\eta(t)) \cdot y) \bar{\mu}(dy)}\right| dt} \notag \\
&\hspace{1cm} \leq \int_0^T{\left|\int{\langle(y - \eta(t)), \nabla \bar{H}(\eta(t)) \rangle \bar{f}(t,y)\bar{\mu}(dy)}\right| dt} \notag \\
&\hspace{3cm}+ \int_0^T{\left|\int{\langle(y - \eta(t)), \nabla \bar{H}(\eta(t)) \rangle \bar{G}(t,y)\bar{\mu}(dy)}\right| dt} \notag \\
&\hspace{1cm} \leq \sqrt{\int_0^T{\int{|y - \eta(t)|_Y^2\bar{f}(t,y)\bar{\mu}(dy)}dt}} \sqrt{\int_0^T{|\nabla \bar{H}(\eta(t))|_Y^2dt}} \notag \\
&\hspace{3cm}+ \sqrt{\int_0^T{\int{|y - \eta(t)|_Y^2\bar{G}(t,y)\bar{\mu}(dy)}dt}} \sqrt{\int_0^T{|\nabla \bar{H}(\eta(t))|_Y^2dt}} \notag \\
&\hspace{1cm} \leq \sqrt{\frac{2(C_2 + \beta) \hspace{0.1cm} \Xi(T,M,N)}{\lambda \tau}} + \sqrt{\frac{(C_2 + \beta)M}{N\lambda \tau}}
\end{align}
which was the last element needed to get the full bound (b).

\end{proof}

\section{Application to Kawasaki Dynamics}

\subsection{Proof of Theorem \ref{30}}

\noindent We shall now prove Theorem \textbf{\ref{30}} as a consequence of Corollary \textbf{\ref{31}}. We consider, in the same way as in [GOVW], a sequence of step functions $\bar{\eta}_{0,\ell} \in \bar{Y}_{\ell}$ such that 
$$||\bar{\eta}_{0,\ell} - \zeta_0||_{\textsl{L}^2} \longrightarrow 0,$$
canonically associate to each of them a vector $\eta_{0,l} \in Y_{\ell}$ and consider the solution $\eta_{\ell}$ of 
$$\frac{d\eta_{\ell}}{dt} = -\bar{A}\nabla \bar{H}(\eta_{\ell}), \hspace{1cm} \eta_{\ell}(0) = \eta_{0,\ell}.$$ We also assume (\ref{601}), that is 
\begin{equation}  N_{\ell} \rightarrow \infty; \hspace{1cm} M_{\ell} \rightarrow \infty; \hspace{1cm} \frac{M_{\ell}}{N_{\ell}} \rightarrow 0;
\end{equation}
which, in this setting, will imply (\ref{602}).

The following proposition was proven in [GOVW]: 

\begin{prop} \label{32}
With the above notations, the step functions $\bar{\eta}_{\ell}$ converge strongly in $\textsl{L}^{\infty}(\textsl{H}^{-1})$ to the unique weak solution of 
$$\frac{\partial \zeta}{\partial t} = \frac{\partial^2}{\partial \theta^2}\varphi'(\zeta), \hspace{1cm} \zeta(0,\cdot) = \zeta_0.$$
\end{prop}

We first have to check that the assumptions of Theorem \textbf{\ref{4}} hold with uniform constants. It has already been checked in [GOVW] that this is the case for assumptions $(i)$ to $(vii)$, so we just have to check assumptions $(viii)$ and $(ix)$.

$\bullet$ It is easy to compute the spectral values of $A$, since it is a circulant matrix. We have 
$$\text{Sp(A)} = \left\{2N^2\left(1 - \cos(\frac{2k\pi}{N})\right), \hspace{0.5cm} k = 0,..,N-1\right\}$$
as an operator on $\mathbb{R}^N$, and the spectral value $0$ corresponds to the action of $A$ on $\mathbb{R}(1,..,1)$, which we don't take into account, since we only consider the action of $A$ on the hyperplane of mean $0$. The lowest spectral value of $A$ is then 
$$\inf \text{Sp(A)} = 2N^2\left(1 - \cos(\frac{2\pi}{N})\right) \underset{N \uparrow \infty}{\longrightarrow} 4\pi^2 > 0.$$
Since the sequence of lowest spectral values converges to a strictly positive limit, we have a strictly positive lower bound on the whole sequence, which proves assumption $(viii)$ with a uniform constant $\tau$.

$\bullet$ Since, by Proposition \textbf{\ref{24}}, $\psi_K''$ converges uniformly to $\varphi''$ and $(\text{Hess} \bar{H}(y))_{ij} = \psi_K''(y_i)\delta_{ij}$, to prove assumption $(ix)$ with a uniform constant, we just have to prove that $\varphi''$  is bounded above. This was actually already proved in ([GOVW], Lemma 41), where is proved both a lower and an upper bound on the second derivative of $\varphi^*(\sigma) = \log \int{\exp(\sigma x - \psi(x))dx}$. It is a property of the Legendre transform that, if $f$ is a strictly convex function, its Legendre transform $f^*$ satisfies $(f^*)^* = f$ and $(f^*)' = (f')^{-1}$, so the strictly positive upper and lower bounds on the second derivative of $\varphi^*$ translate into strictly positive upper and lower bounds on $\varphi''$.

We will also check that, in this case, (\ref{601}) implies (\ref{602}). This is easy to check: since we have $|y|_Y^2 = \frac{1}{M}\sum y_i^2$, we can explicitly compute 
\begin{equation} \label{666}
\Gamma(Y, |\cdot |_Y) = (\sqrt{2\pi M})^{M-1},
\end{equation}
and (\ref{602}) follows.

We can therefore apply Corollary \textbf{\ref{31}}. (\ref{ent1}) is then a direct application of (a'), so we will concentrate on the proof of (\ref{ent2}). Part (b') of Corollary \textbf{\ref{31}} states that
$$\int_0^T{\left|\frac{1}{N}\int{\Phi(f_{\ell}(t,x))\mu_{\ell}(dx)} - \bar{H_{\ell}}(\eta_{\ell}(t))\right|dt} \longrightarrow 0,$$
so we now just have to prove that
\begin{equation} \label{201}
\int_0^T{\left|\bar{H_{\ell}}(\eta_{\ell}(t)) - \int_{\mathbb{T}}{\varphi(\zeta(\theta,t))d\theta} + \varphi\left(\int_{\mathbb{T}}{\zeta(t,\theta) d\theta}\right)\right|dt} \underset{\ell \uparrow \infty}{\longrightarrow} 0.
\end{equation}
We have the expression 
\begin{align}
\bar{H}(y) &= \frac{1}{M}\sum \psi_K(y_i) +\frac{1}{N}\log \bar{Z} \notag \\
&= \int_{\mathbb{T}}{\psi_K(\bar{y})d\theta} +\frac{1}{N}\log \bar{Z}
\end{align}
As a consequence of Proposition \textbf{\ref{32}}, we shall prove that $\int_{\mathbb{T}}{\psi_K(\bar{\eta})d\theta}$ converges to $\int_{\mathbb{T}}{\varphi(\zeta(\theta,t))d\theta}$ in a time-integrated sense, and then we shall prove that $\frac{1}{N}\log \bar{Z}$ converges to $-\varphi(\int_{\mathbb{T}}{\zeta(t,\theta)d\theta})$, which will yield (\ref{201}). By the triangle inequality
\begin{align} &\left|\int_{\mathbb{T}}{\psi_K(\bar{\eta}(\theta,t))d\theta } - \int_{\mathbb{T}}{\varphi(\zeta(\theta,t))d\theta}\right| \notag \\
&\hspace{0.5cm} \leq \int_{\mathbb{T}}{|\psi_K(\bar{\eta_{\ell}}(t,\theta)) - \varphi(\bar{\eta_{\ell}}(t,\theta))|d\theta} + \int_{\mathbb{T}}{|\varphi(\bar{\eta_{\ell}}(t,\theta)) - \varphi(\zeta(t, \theta))|d\theta}.
\end{align}
But 
$$\int_{\mathbb{T}}{|\psi_K(\bar{\eta_{\ell}}(t,\theta)) - \varphi(\bar{\eta_{\ell}}(t,\theta))|d\theta} \leq ||\psi_K - \varphi||_{\infty} \underset{K \rightarrow \infty}{\longrightarrow} 0$$
and by convexity, and since $\varphi'' \leq \Lambda$,
\begin{align} \varphi'(\zeta(t, \theta))\left(\bar{\eta_{\ell}}(t,\theta) - \zeta(t, \theta)\right) &\leq \varphi(\bar{\eta_{\ell}}(t,\theta)) - \varphi(\zeta(t, \theta)) \notag \\
&\leq \varphi'(\zeta(t, \theta))(\bar{\eta_{\ell}}(t,\theta) - \zeta(t, \theta)) + \frac{\Lambda}{2}\left|\bar{\eta}_{\ell}(t, \theta) - \zeta(t, \theta)\right|^2
\end{align}
We know that $\bar{\eta}$ converges to $\zeta$ in $\textsl{L}^{\infty}(H^{-1})$. Since
$$\int_0^T{\int_{\mathbb{T}}{\left|\varphi'(\zeta(t, \theta))(\bar{\eta_{\ell}}(t,\theta) - \zeta(t, \theta))\right|d\theta}dt} \leq \sqrt{\int_0^T{||\varphi'(\rho(t))||_{H^1}^2dt}}\sqrt{\int_0^T{||\bar{\eta_{\ell}}(t) - \zeta(t)||_{H^{-1}}^2dt}}$$
and $\varphi'(\zeta) \in \textsl{L}^2(H^1)$
we deduce that $\int_{\mathbb{T}}{|\varphi(\bar{\eta_{\ell}}(t,\theta)) - \varphi(\zeta(t, \theta))|d\theta}$ converges to 0 in a time-integrated sense, and thus
$$\int_0^T{\left|\int_{\mathbb{T}}{\psi_K(\bar{\eta_{\ell}}(\theta,t))d\theta } - \int_{\mathbb{T}}{\varphi(\zeta(\theta,t))d\theta}\right|dt} \longrightarrow 0.$$

Note that, if we have a time-uniform bound on $||\varphi'(\rho(t))||_{H^1}$, this convergence actually holds uniformly in time. In the proof of Proposition \ref{reg_hydro}, we show that such a bound holds on time intervals $[\epsilon, +\infty[$, for any $\epsilon > 0$.

To conclude the proof of (\ref{201}), it is enough to prove that 
\begin{equation} \label{55}
\frac{1}{N}\log \bar{Z} \longrightarrow \varphi\left(\int_{\mathbb{T}}{\zeta(t,\theta)d\theta}\right)
\end{equation}
uniformly in time. 

First of all, recall that at any time $t$ we have $\int_{\mathbb{T}}{\zeta(t,\theta)d\theta} = m$. Also recall that
$$\bar{Z} = \int{\exp\left(-\frac{N}{M}\underset{i=1}{\stackrel{M}{\sum}}\psi_K(y_i)\right)dx}.$$
Since $y \mapsto \frac{1}{M}\underset{i=1}{\stackrel{M}{\sum}}\psi_K(y_i)$ is strictly convex, it has a unique minimum on $Y$, and since the variables are exchangeable this minimum can only be reached for $y_1 = ... = y_M$, and by definition of $Y$ this can only be the case if all the $y_i$ are equal to $m$.
Since $\lambda \leq \psi_K'' \leq \Lambda$ and $||y||_Y = \frac{1}{M}\underset{i=1}{\stackrel{M}{\sum}}|y_i|^2$, by convexity, for all $y \in Y$ we have
$$\psi_K(m) + \frac{\lambda}{2M}||y - m||_2^2 \leq \frac{1}{M}\underset{i=1}{\stackrel{M}{\sum}}\psi_K(y_i) \leq \psi_K(m) + \frac{\Lambda}{2M}||y - m||_2^2$$
where $|| \cdot ||_2$ is the usual Euclidean norm, and we identify the mean $m$ and the vector of $Y$ where all coordinates are equal to $m$. We take the exponential of this inequality multiplied by $-N$ and integrate, which gives us, since for any $y \in Y$, $y - m$ is of mean $0$,
$$-\inf \psi_K + \frac{1}{N}\log \int_{\mathbb{R}^{M-1}}{\exp\left(-\frac{\Lambda N}{2M}||y||_2^2\right)dy} \leq \frac{1}{N} \log \bar{Z}$$
$$\leq -\inf \psi_K + \frac{1}{N}\log \int_{\mathbb{R}^{M-1}}{\exp\left(-\frac{\lambda N}{2M}||y||_2^2\right))dy}.$$
Since 
$$\frac{1}{N}\log \int_{\mathbb{R}^{M-1}}{\exp\left(-\frac{\Lambda N}{2M}||y||_2^2\right)dy} = \frac{M-1}{2N}\log\left(\frac{\Lambda N}{2\pi}\right) \rightarrow 0$$
and the same goes for $\frac{1}{N}\log \int{\exp(-\frac{\lambda N}{2M}||y||_2^2)dy}$, we deduce that $|\frac{1}{N} \log \bar{Z} + \inf \psi_K|$ goes to $0$ uniformly in time. Finally, since $\psi_K$ converges uniformly to $\varphi$, $\psi_K(m)$ converges to $\varphi(m)$, which implies the desired result.

\subsection{Proof of Theorem \ref{77}}

We shall now use the time-integrated convergence of the entropy we just proved to show that this convergence actually holds pointwise. Our proof closely follows an idea of [K]. This method was pointed out to us by the (anonymous) referee. It is also possible to deduce the pointwise convergence from the time integrated convergence by using the relative entropy method devised in [Y], but this yields a much longer proof.

In a first step, we will show pointwise convergence of the entropy, by showing that 
\begin{equation} \label{liminf}
\liminf \hspace{1mm} \frac{1}{N}\operatorname{Ent}_{\mu_N}(f_N(t)) \geq \int{\varphi(\zeta(t,\theta))d\theta} - \varphi\left(\int{\zeta(t,\theta)d\theta} \right)
\end{equation}
and 
\begin{equation} \label{limsup}
\limsup \hspace{1mm} \frac{1}{N}\operatorname{Ent}_{\mu_N}(f_N(t)) \leq \int{\varphi(\zeta(t,\theta))d\theta} - \varphi\left(\int{\zeta(t,\theta)d\theta} \right).
\end{equation}
In a second step, we will show that this pointwise convergence actually holds uniformly in time, as long as we stay away from time $t = 0$.

Let us start with the upper bound. We know that 
$$\frac{d}{dt}\int{f(t,x)\log f(t,x) \mu(dx)} = -\int{\frac{\langle A\nabla f, \nabla f \rangle}{f}d\mu} \leq 0,$$
so that, for any $N$, the entropy $\frac{1}{N}\operatorname{Ent}_{\mu_N}(f_N)$ is decreasing in time. This is just the H-theorem expressed in the context of our model.

Therefore, for any $N$, any $t > 0$ and $\epsilon$ small enough, we have

\begin{equation} \label{referee}
\frac{1}{N}\operatorname{Ent}_{\mu_N}(f_N(t)) \leq \frac{1}{\epsilon}\int_{t - \epsilon}^t{\frac{1}{N}\operatorname{Ent}_{\mu_N}(f_N(s))ds}
\end{equation}

we know from Theorem \ref{30} that $\int_{t - \epsilon}^t{\frac{1}{N}\operatorname{Ent}_{\mu_N}(f_N(s))ds}$ converges to 

$\int_{t - \epsilon}^t{\int_{\T}{\varphi(\zeta(s,\theta))d\theta} - \varphi \left(\int{\zeta(s,\theta)d\theta}\right) ds}$. Therefore, for any $t > 0$ and any $\epsilon$ small enough, we have
$$\limsup \frac{1}{N}\operatorname{Ent}_{\mu_N}(f_N(t)) \leq \frac{1}{\epsilon}\int_{t - \epsilon}^t{\int_{\T}{\varphi(\zeta(s,\theta))d\theta} - \varphi \left(\int{\zeta(s,\theta)d\theta}\right) ds}.$$

Since $\rho$ is smooth, by Proposition \ref{reg_hydro}, letting $\epsilon$ go to zero yields (\ref{limsup}). (\ref{liminf}) can be obtained in the same way, by using the inequality
$$\frac{1}{N}\operatorname{Ent}_{\mu_N}(f_N(t)) \geq \frac{1}{\epsilon}\int_t^{t + \epsilon}{\frac{1}{N}\Ent_{\mu_N}(f_N(s))ds}.$$

Since the functions $t \rightarrow \frac{1}{N}\operatorname{Ent}_{\mu_N}(f_N(t))$ are continuous and decreasing, and the function $t \rightarrow \int{\varphi(\zeta(t,\theta))d\theta} - \varphi\left(\int{\zeta(t,\theta)d\theta} \right)$ is continuous, Dini's second theorem implies that this pointwise convergence is actually uniform on the compact sets $[\epsilon, T]$, for any $T > \epsilon > 0$.

\subsection{Proof of Proposition \ref{reg_hydro}}

To prove the regularity of the solution of the hydrodynamic equation, we shall use the following interpolation inequality, which is a particular case of a family of inequalities that can be found in the second chapter of [LSU].

\begin{lem} \label{lem_reg}
For any $u \in H^1(\mathbb{T})$ with $\int_{\mathbb{T}}{u d\theta} = 0$ we have
$$||u||_{L^4} \leq 2^{1/4}||u||_{L^2}^{3/4} ||u'||_{L^2}^{1/4}.$$
\end{lem}

\begin{proof}
Let us take such a function $u$. We have
$$|u(\theta)|^4 = |u(\theta)|^2 |u(\theta)|^2  \leq |u(\theta)|^2 \left( \int_{\mathbb{T}}{2 |u(s)| \hspace{1mm} |u'(s)|ds} \right)$$
Using H\"older's inequality, we have
$$\int_{\mathbb{T}}{|u(s)| \hspace{1mm} |u'(s)|ds} \leq ||u||_{L^2} ||u'||_{L^2},$$
so that
$$\int_{\mathbb{T}}{|u(\theta)|^4 d\theta} \leq 2||u||_{L^2}^3 ||u'||_{L^2}$$
and the result immediately follows.
\end{proof}

To prove the regularity of our function, we shall prove bounds on the $L^2$ norms of the derivatives of $\varphi'(\zeta)$, using differential inequalities, and then Sobolev injections. We first have

\begin{align} \label{der1}
\frac{d}{dt} \int{\zeta(t,\theta)^2d\theta} &= 2\int{\zeta \frac{\partial^2}{\partial \theta^2} \varphi'(\zeta)d\theta} \notag \\
&= - 2\int{\left( \frac{\partial \varphi'(\zeta)}{\partial \theta} \right)\frac{\partial \zeta}{\partial \theta}d\theta} \notag \\
&= -2\int{\varphi''(\zeta)\left(\frac{\partial \zeta}{\partial \theta} \right)^2d\theta} \notag \\
&\leq -2(\inf \varphi'')\int{\left(\frac{\partial \zeta}{\partial \theta} \right)^2d\theta}
\end{align}

Integrating this inequality yields

\begin{align}
\int_0^T{\int{\left( \frac{\partial \zeta}{\partial \theta} \right)^2d\theta}dt} &\leq \frac{1}{2(\inf \varphi'')} \left(||\zeta(0,\cdot)||_{L^2}^2 - ||\zeta(T,\cdot)||_{L^2}^2\right) \notag \\
&\leq \frac{1}{2(\inf \varphi'')} ||\zeta(0,\cdot)||_{L^2}^2.
\end{align}
Since $||\zeta(0,\cdot)||_{L^2}^2$ is finite, we obtain
$$\int_0^\infty{\int{\left( \frac{\partial \zeta}{\partial \theta} \right)^2d\theta}dt} < \infty.$$
Moreover, since $\varphi''$ is bounded, we also get the bound on $\partial \varphi'(\zeta)/\partial \theta = \varphi''(\zeta)\partial \zeta/\partial \theta$ : 
\begin{equation} \label{reg1}
\int_0^\infty{\int{\left( \frac{\partial \varphi'(\zeta)}{\partial \theta} \right)^2d\theta}dt} < \infty.
\end{equation}

We then have
\begin{align} 
\frac{1}{2}\frac{d}{dt} \int{\left( \frac{\partial \varphi'(\zeta)}{\partial \theta} \right)^2d\theta} &= \int{ \frac{\partial \varphi'(\zeta)}{\partial \theta} \frac{\partial}{\partial \theta} \left( \varphi''(\zeta)\frac{\partial^2}{\partial \theta^2}\varphi'(\zeta) \right)d\theta} \notag \\
&= - \int{\varphi''(\zeta)\left(\frac{\partial^2}{\partial \theta^2}\varphi'(\zeta) \right)^2d\theta} \notag \\
&\leq -\lambda \int{\left(\frac{\partial^2}{\partial \theta^2}\varphi'(\zeta) \right)^2d\theta} \label{extra} \\
&\leq -\lambda \pi^2 \int{\left(\frac{\partial}{\partial \theta}\varphi'(\zeta) \right)^2d\theta} \label{reg2}
\end{align}
where the last inequality is a consequence of the Poincar\'e inequality $||u||_{L^2} \leq \pi||u'||_{L^2}$ for all functions in $H^1(\mathbb{T})$ with mean zero.

Combining (\ref{reg1}) and (\ref{reg2}), we get for any $t_1 > t_2 > 0$
\begin{equation} \label{reg9}
\int_{\mathbb{T}}{\left( \frac{\partial \varphi'(\zeta(t_2,\theta))}{\partial \theta} \right)^2 d\theta} \leq \frac{C}{t_1} \exp(-2\lambda \pi^2(t_2 - t_1)).
\end{equation}

Moreover, using (\ref{extra}), we get
\begin{equation} \label{reg8}
\int_{\epsilon}^T{\int{\left(\frac{\partial^2}{\partial \theta^2}\varphi'(\zeta) \right)^2d\theta}} \leq \frac{CT}{\epsilon}
\end{equation}
for any $0 < \epsilon < T$.

In the same way, we have 
\begin{align} \label{reg6}
\frac{1}{2}\frac{d}{dt} &\int{\left( \frac{\partial^2 \varphi'(\zeta)}{\partial \theta^2} \right)^2d\theta} = - \int{\left(\frac{\partial^3 \varphi'(\zeta)}{\partial \theta^3}\right) \frac{\partial}{\partial \theta} \left(\varphi''(\zeta) \frac{\partial^2 \varphi'(\zeta)}{\partial \theta^2} \right)} \notag \\
&= -\int{\varphi''(\zeta)\left( \frac{\partial^3 \varphi'(\zeta)}{\partial \theta^3} \right)^2d\theta} - \int{\left( \frac{\partial^3 \varphi'(\zeta)}{\partial \theta^3} \right)\left( \frac{\partial^2 \varphi'(\zeta)}{\partial \theta^2} \right)\left(\frac{\partial \varphi''(\zeta)}{\partial \theta} \right) d\theta} \notag \\
&\leq -\lambda \int{\left( \frac{\partial^3 \varphi'(\zeta)}{\partial \theta^3} \right)^2d\theta} - \int{\left( \frac{\partial^3 \varphi'(\zeta)}{\partial \theta^3} \right)\left( \frac{\partial^2 \varphi'(\zeta)}{\partial \theta^2} \right)\left(\frac{\partial \varphi''(\zeta)}{\partial \theta} \right) d\theta}.
\end{align}

A simple calculation yields
$$\frac{\partial \varphi''(\zeta)}{\partial \theta} = \frac{\varphi^{(3)}(\zeta)}{\varphi''(\zeta)} \frac{\partial \varphi'(\zeta)}{\partial \theta}$$
and our assumption of boundedness on $\varphi^{(3)}/\varphi''$ then yields the bound
$$\left| \int{\left( \frac{\partial^3 \varphi'(\zeta)}{\partial \theta^3} \right)\left( \frac{\partial^2 \varphi'(\zeta)}{\partial \theta^2} \right)\left(\frac{\partial \varphi''(\zeta)}{\partial \theta} \right) d\theta} \right|$$

$$\leq C\int{\left| \left(\frac{\partial^3 \varphi'(\zeta)}{\partial \theta^3} \right)\left( \frac{\partial^2 \varphi'(\zeta)}{\partial \theta^2} \right)\left(\frac{\partial \varphi'(\zeta)}{\partial \theta} \right)\right| d\theta}.$$

Using H\"older's inequality, we then have
\begin{align} \label{reg3}
&\left| \int{\left( \frac{\partial^3 \varphi'(\zeta)}{\partial \theta^3} \right)\left( \frac{\partial^2 \varphi'(\zeta)}{\partial \theta^2} \right)\left(\frac{\partial \varphi''(\zeta)}{\partial \theta} \right) d\theta} \right| \notag \\
&\leq \left( \int{\left(\frac{\partial^3 \varphi'(\zeta)}{\partial \theta^3} \right)^2d\theta} \right)^{1/2} \left( \int{\left(\frac{\partial^2 \varphi'(\zeta)}{\partial \theta^2} \right)^4d\theta} \right)^{1/4} \left( \int{\left(\frac{\partial \varphi'(\zeta)}{\partial \theta} \right)^4d\theta} \right)^{1/4}
\end{align}

By an application of Lemma \ref{lem_reg}, we have

\begin{equation} \label{reg4}
\left( \int{\left(\frac{\partial^2 \varphi'(\zeta)}{\partial \theta^2} \right)^4d\theta} \right)^{1/4} \leq C\left( \int{\left(\frac{\partial^2 \varphi'(\zeta)}{\partial \theta^2} \right)^2d\theta} \right)^{3/8} \left( \int{\left(\frac{\partial^3 \varphi'(\zeta)}{\partial \theta^3} \right)^2d\theta} \right)^{1/8}
\end{equation}

and 

\begin{equation} \label{reg5}
\left( \int{\left(\frac{\partial \varphi'(\zeta)}{\partial \theta} \right)^4d\theta} \right)^{1/4} \leq C\left( \int{\left(\frac{\partial \varphi'(\zeta)}{\partial \theta^2} \right)d\theta} \right)^{3/8} \left( \int{\left(\frac{\partial^2 \varphi'(\zeta)}{\partial \theta^2} \right)^2d\theta} \right)^{1/8}.
\end{equation}

Plugging (\ref{reg4}) and (\ref{reg5}) into (\ref{reg3}), we get

\begin{align} 
&\left| \int{\left( \frac{\partial^3 \varphi'(\zeta)}{\partial \theta^3} \right)\left( \frac{\partial^2 \varphi'(\zeta)}{\partial \theta^2} \right)\left(\frac{\partial \varphi''(\zeta)}{\partial \theta} \right) d\theta} \right| \notag \\
&\leq \left( \int{\left(\frac{\partial^3 \varphi'(\zeta)}{\partial \theta^3} \right)^2d\theta} \right)^{5/8} \left( \int{\left(\frac{\partial^2 \varphi'(\zeta)}{\partial \theta^2} \right)^2 d\theta} \right)^{1/2} \left( \int{\left(\frac{\partial \varphi'(\zeta)}{\partial \theta} \right)^2 d\theta} \right)^{3/8}.
\end{align}
Using the classical interpolation inequality $||u'||_{L^2}^2 \leq ||u||_{L^2}||u''||_{L^2}$, we get
$$\left( \int{\left(\frac{\partial^2 \varphi'(\zeta)}{\partial \theta^2} \right)^2 d\theta} \right)^{1/2} \leq \left( \int{\left(\frac{\partial^3 \varphi'(\zeta)}{\partial \theta^3} \right)^2d\theta} \right)^{1/4} \left( \int{\left(\frac{\partial \varphi'(\zeta)}{\partial \theta} \right)^2 d\theta} \right)^{1/4}$$

and therefore

\begin{align} 
&\left| \int \left( \frac{\partial^3 \varphi'(\zeta)}{\partial \theta^3} \right)\left( \frac{\partial^2 \varphi'(\zeta)}{\partial \theta^2} \right)\left(\frac{\partial \varphi''(\zeta)}{\partial \theta} \right) d\theta \right| \notag \\
& \hspace{3cm} \leq \left( \int{\left(\frac{\partial^3 \varphi'(\zeta)}{\partial \theta^3} \right)^2d\theta} \right)^{7/8} \left( \int{\left(\frac{\partial \varphi'(\zeta)}{\partial \theta} \right)^2 d\theta} \right)^{5/8}.
\end{align}

Finally, using Young's inequality $ab \leq 7a^{8/7}/8 + b^8/8$, we get for any $\delta > 0$
\begin{align} 
&\left| \int{\left( \frac{\partial^3 \varphi'(\zeta)}{\partial \theta^3} \right)\left( \frac{\partial^2 \varphi'(\zeta)}{\partial \theta^2} \right)\left(\frac{\partial \varphi''(\zeta)}{\partial \theta} \right) d\theta} \right| \notag \\
& \hspace{15mm} \leq C \delta^{8/7} \left( \int{\left(\frac{\partial^3 \varphi'(\zeta)}{\partial \theta^3} \right)^2d\theta} \right) + \frac{C}{\delta^8} \left( \int{\left(\frac{\partial \varphi'(\zeta)}{\partial \theta} \right)^2 d\theta} \right)^5.
\end{align}

Taking $\delta$ small enough and inserting this inequality into (\ref{reg6}), we get
\begin{align} \label{reg7}
\frac{1}{2}\frac{d}{dt} &\int{\left( \frac{\partial^2 \varphi'(\zeta)}{\partial \theta^2} \right)^2d\theta} \notag \\
&\leq -\frac{\lambda}{2} \int{\left( \frac{\partial^3 \varphi'(\zeta)}{\partial \theta^3} \right)^2d\theta} + C\left( \int{\left(\frac{\partial \varphi'(\zeta)}{\partial \theta} \right)^2 d\theta} \right)^5 \notag \\
&\leq -\frac{\lambda}{2\pi^2} \int{\left( \frac{\partial^2 \varphi'(\zeta)}{\partial \theta^2} \right)^2d\theta} + C\left( \int{\left(\frac{\partial \varphi'(\zeta)}{\partial \theta} \right)^2 d\theta} \right)^5.
\end{align}

Combining (\ref{reg9}), (\ref{reg8}) and (\ref{reg7}), it is easy to see that $\int{\left( \frac{\partial^2 \varphi'(\zeta)}{\partial \theta^2} \right)^2d\theta}$ is uniformly bounded for $t$ in $[\epsilon, T]$, for all $T > \epsilon > 0$. Since we can inject $H^2(\mathbb{T})$ into $C^{1+\alpha}(\mathbb{T})$ for some $\alpha > 0$, $\varphi'(\zeta(t,\cdot))$ lies in $C^{1+\alpha}(\mathbb{T})$ for all $t$ in $[\epsilon, T]$. Since $\varphi'$ is invertible and $\varphi''$ is positive, this implies that $\zeta(t,\cdot)$ also lies in $C^{1+\alpha}(\mathbb{T})$ for all $t$ in $[\epsilon, T]$. Using this fact, we can rewrite the PDE as
$$\frac{\partial \zeta}{\partial t} = \varphi''(\zeta)\frac{\partial^2 \zeta}{\partial \theta^2} + \varphi^{(3)}(\zeta) \left(\frac{\partial \zeta}{\partial \theta} \right)^2.$$

Taking $a(t,\theta) = \varphi''(\zeta(t,\theta))$ and $b(t,\theta) = \varphi^{(3)}(\zeta) \left(\frac{\partial \zeta}{\partial \theta} \right)$, we get that $\zeta$ is a solution of the linear parabolic PDE
$$\frac{\partial \zeta}{\partial t} = a(t,\theta)\frac{\partial^2 \zeta}{\partial \theta^2} + b(t,\theta) \frac{\partial \zeta}{\partial \theta}$$
with coefficients $a$ and $b$ that belong to $C^{\alpha}$. We can then use the theory for regularity of the solutions of linear parabolic equations (see for example [LSU]) to show that $\zeta(t,\cdot)$ lies in $C^{2+\alpha}(\mathbb{T})$ for all $t$ in $[\epsilon, T]$. The fact that $\frac{\partial \zeta}{\partial t}$ lies in $C^{\alpha}$ for all $t$ in $[\epsilon, T]$ immediately follows from the PDE.

\vspace{1cm}

{\Large \textbf{Acknowledgments}}

I would like to thank C\'edric Villani for having started me on this topic and for his tremendous help, as well as S.R.S Varadhan for his help, and Maria G. Westdickenberg for her advice and proofreading. I would also like to thank the (anonymous) referee, who sugggested using inequality (\ref{referee}) to prove Theorem \ref{77}, rather than the relative entropy method, which greatly simplifies the proof.

\vspace{1cm}

{\Large \textbf{Bibliography}}
{\small
\begin{itemize}

\item
\label{Go}[Go]
	Gozlan, N.,
	A Characterization of Dimension-Free Concentration in Terms of Transportation Inequalities
	\textit{Ann. Probab.} \textbf{37}, Number 6 (2009), 2480-2498.

\item
\label{GOVW}[GOVW]
	Grunewald, N., Otto, F., Villani, C. and Westdickenberg, M. G.,
	A two-scale approach to logarithmic Sobolev inequalities and the hydrodynamic limit.
	\textit{Ann. Inst. H. Poincar\'e Probab. Statist}. 
	45 (2009), 2, 302--351.

\item
\label{GPV}[GPV]
	Guo, M.Z., Papanicolau, G.C. and Varadhan, S.R.S.
	Nonlinear Diffusion Limit for a System with Nearest Neighbor Interactions,
	\textit{Commun. Math. Phys.} 118, 31-59 (1988)

\item
\label{K}[K]
	Kosygina, E.,
	The Behavior of the Specific Entropy in the Hydrodynamic Scaling Limit for the Ginzburg-Landau Model,
	\textit{Markov Processes and Related Fields}, 
	\textbf{7}, 3 (2001), pp. 383-417.

\item
\label{L}[L]
	Ledoux, M.,
	The Concentration of Measure Phenomenon,
	AMS, Math. Surveys and Monographs, \textbf{89}, Providence, Rhode Island, 2001.
	
\item
\label{LSU}[LSU]
	Ladyzenskaja, O. A., Solonnikov, V. A. and Uralceva, N. N., 
	Linear and Quasilinear Equations of Parabolic Type.
	Translations of Mathematical Monographs, Volume 23, AMS, 1968.
	
\item
\label{MO}[MO]
	Otto, F. and Menz, G.,
	Uniform logarithmic Sobolev inequalities for conservative spin systems with super-quadratic single-site potential.
	To appear in \textit{Ann. Probab.} (2011)
	
\item
\label{OV}[OV]
	Otto, F. and Villani, C.,
	Generalization of an Inequality by Talagrand and Links with the Logarithmic Sobolev Inequality,
	\textit{J. Funct. Analysis}, \textbf{243} (2007), pp. 121-157.

\item
\label{Y}[Y]
	Yau, H.T.,
	Relative Entropy and Hydrodynamics of Ginzburg-Landau Models, 
	\textit{Lett. Math. Phys.}, \textbf{22} (1991), 63-80.
\end{itemize}
}

\end{document}